\documentclass[leqno,12pt]{amsart} 
\usepackage{amssymb}
\usepackage{amsmath}
\usepackage{amsthm}
\usepackage{amscd}
\usepackage{graphicx}
\usepackage{url}
\usepackage[all]{xy}
\usepackage[dvips]{color}

\newcommand{\mathsym}[1]{{}}
\newcommand{\unicode}[1]{{}}
\newcommand{\del}{{\partial}}
\newcommand{\delbar}{\overline{\partial}}
\newcommand{\ddbar}{\partial\overline{\partial}}
\newcommand{\ve}{\varepsilon}
\newcommand{\vp}{\varphi}
\theoremstyle{plain} 
\newtheorem{theorem}{\indent\sc Theorem}[section]
\newtheorem{lemma}[theorem]{\indent\sc Lemma}
\newtheorem{corollary}[theorem]{\indent\sc Corollary}
\newtheorem{proposition}[theorem]{\indent\sc Proposition}

\theoremstyle{definition} 

\newtheorem{remark}[theorem]{\indent\sc Remark}
\newtheorem{example}[theorem]{\indent\sc Example}

\newtheorem{question}[theorem]{\indent\sc Question}
\begin{document}

\title[$\delbar$ cohomology of the complement of a semi-positive anticanonical divisor]
{$\delbar$ cohomology of the complement of a semi-positive anticanonical divisor of a compact surface} 

\author[T. Koike]{Takayuki Koike} 

%
\address{
Department of Mathematics, Graduate School of Science, Osaka Metropolitan University \endgraf
3-3-138, Sugimoto, Sumiyoshi-ku Osaka, 558-8585 \endgraf
Japan
}
\email{tkoike@omu.ac.jp}

\maketitle

\begin{abstract}
Let $X$ be a non-singular compact complex surface such that the anticanonical line bundle admits a smooth Hermitian metric with semi-positive curvature. 
For a non-singular hypersurface $Y$ which defines an anticanonical divisor, 
we investigate the $\delbar$ cohomology group $H^1(M, \mathcal{O}_M)$ of the complement $M=X\setminus Y$. 
\end{abstract}

\section{Introduction}

Let $X$ be a connected compact K\"ahler surface such that the anticanonical line bundle $K_X^{-1}:=\Lambda^2T_X$ is {\it semi-positive}: i.e. there exists a $C^\infty$ Hermitian metric $h$ on $K_X^{-1}$ such that the Chern curvature $\sqrt{-1}\Theta_h$ is positive semi-definite at each point of $X$. 
Our interest is in the complex analytic structure of the complement $M$ of the support of an anticanonical divisor $D\in |K_X^{-1}|$. 
In this paper, we investigate it under the assumption that $D$ is the divisor defined by a non-singular hypersurface $Y\subset X$: i.e. $D=Y_1+Y_2+\cdots +Y_N$, where $\{Y_\nu\}_{\nu=1}^N$ is the set of all the connected components of $Y$. 
Note that the degree ${\rm deg}\,N_{Y_\nu/X}$ of the normal bundle $N_{Y_\nu/X}$ of each $Y_\nu$ is non-negative in this case, since it coincides with $\int_{Y_\nu}\textstyle\frac{\sqrt{-1}}{2\pi}\Theta_h$. 

If ${\rm deg}\,N_{Y_\nu/X}$ is positive for some component $Y_\nu$ of $Y$, 
it follows from Hodge index theorem and \cite[Proposition (2.2)]{S} that $Y$ is connected and that the complement $M$ is strongly $1$-convex (See \S \ref{section:prelim_case_positive} for the details). 
Otherwise, each $N_{Y_\nu/X}$ is an element of the Picard variety ${\rm Pic}^0(Y_\nu)$ of $Y_\nu$. 
When $N_{Y_\nu/X}$ is a torsion element of ${\rm Pic}^0(Y_\nu)$ (i.e. $N_{Y_\nu/X}^{n}:=N_{Y_\nu/X}^{\otimes n}$ is holomorphically trivial for a positive integer $n$) for some component $Y_\nu$ of $Y$, it follows from \cite[Theorem 1.1 (i)]{Ko2020} that there exists an elliptic fibration $f\colon X\to B$ onto a compact Riemann surface $B$ such that $Y_\mu$ is a fiber of $f$ for each $\mu\in\{1, 2, \dots, N\}$ (See \S \ref{section:prelim_case_torsion_fibrartion}). 
Especially, $N_{Y_\mu/X}$ is a torsion element of ${\rm Pic}^0(Y_\mu)$ for each $Y_\mu$ and $M$ admits a proper holomorphic surjection onto an open Riemann surface in this case. 
In the present paper, we mainly investigate the remaining case: i.e. the case where $N_{Y_\nu/X}$ is a non-torsion element of ${\rm Pic}^0(Y_\nu)$ for each $\nu\in \{1, 2, \dots, N\}$. In this case, as will be observed in \S \ref{section:prelim_case_non_tor_1compl_etc}, it follows from \cite{Ko2020} and \cite[Remark 5.2]{O} that $M$ is a weakly $1$-complete manifold which admits a family of compact Levi-flat hypersurfaces. 

As an example, let us consider the case where $X={\bf P}(\mathbb{I}_C\oplus F)$ is a ruled surface over an elliptic curve $C$ obtained by the fiberwise projectivization of the vector bundle $\mathbb{I}_C\oplus F$, where $\mathbb{I}_C$ is the holomorphically trivial line bundle over $C$ and $F$ is a non-torsion element of ${\rm Pic}^0(C)$. 
Note that $K_X^{-1}$ is semi-positive, since the fiberwise Fubini-Study metric induces a $C^\infty$ Hermitian metric on $K_X^{-1}$ with semi-positive curvature. 
For non-singular hypersurfaces $Y_1:={\bf P}(\mathbb{I}_C)$ and $Y_2:={\bf P}(F)$, it can be easily checked that $Y_1+Y_2$ is an anticanonical divisor of $X$, and that $N_{Y_\nu/X}$ is a non-torsion element of ${\rm Pic}^0(Y_\nu)$ for $\nu=1, 2$. 
In this example, $M=X\setminus (Y_1\cup Y_2)$ is a {\it toroidal group} (a complex Lie group without global non-constant holomorphic function). 
By Kazama's theorem \cite{Ka} on the $\delbar$ cohomology groups of toroidal groups, it follows that the $\delbar$ cohomology group $H^1(M, \mathcal{O}_M)$ is either of $1$-dimension or {\it of non-Hausdorff type}. 
Here we say that $H^1(M, \mathcal{O}_M)$ is of non-Hausdorff type 
if it is of infinite dimension and not Hausdorff by regarding it as the topological vector space obtained by the quotient $Z^{0, 1}(M)/B^{0, 1}(M)$, 
where we are regarding the space $A^{p, q}(M)$ of all the smooth $(p, q)$-forms on $M$ as a Fr\'echet space by using the topology of uniform convergence on compact sets in all the derivatives of coefficient functions, and topologizing $B^{0, 1}(M):={\rm Image}(\delbar\colon A^{0, 0}(M)\to A^{0, 1}(M))$ and $Z^{0, 1}(M):={\rm Ker}(\delbar\colon A^{0, 1}(M)\to A^{0, 2}(M))$ by using the relative topologies as subspaces of $A^{0, 1}(M)$. 
Moreover, Kazama also gives a sufficient and necessary condition for ${\rm dim}\,H^1(M, \mathcal{O}_M)<\infty$ by using an irrational number theoretical condition for the discrete subgroup of the Euclidean space which corresponds to the toroidal group \cite[Theorem 4.3]{Ka} (See also \cite{KaU}, \cite[\S 2.2]{AK}). 
Note that a non-Hausdorffness of the $L^2$ $\delbar$ cohomology is shown on a certain kind of a Stein domain in \cite{CS}. 

In the present paper, as a generalization of Kazama's result for surfaces, we show the following: 
\begin{theorem}\label{thm:main}
Let $X$ be a non-singular connected compact complex surface such that $K_X^{-1}$ is semi-positive, and $Y\subset X$ be a non-singular hypersurface such that $Y_1+Y_2+\cdots+Y_N$ is an anticanonical divisor, where $\{Y_\nu\}_{\nu=1}^N$ is the set of all the connected components of $Y$. Assume that $N_{Y_\nu/X}$ is a non-torsion element of ${\rm Pic}^0(Y_\nu)$ for each $\nu\in \{1, 2, \dots, N\}$. 
Then, for the complement $M:=X\setminus Y$, either of the following two conditions holds: \\
$(i)$ ${\rm dim}\,H^1(M, \mathcal{O}_M) = {\rm dim}\,H^1(X, \mathcal{O}_X)$, or\\ 
$(ii)$ $H^1(M, \mathcal{O}_M)$ is of non-Hausdorff type: i.e. $H^1(M, \mathcal{O}_M)$ is of infinite dimension and not Hausdorff with respect to the topology as above. \\
Moreover, the condition $(i)$ holds if and only if there exist positive numbers $\ve$ and $\delta$ such that, for any $\nu\in\{1, 2, \dots, N\}$ and any positive integer $n$, $d(\mathbb{I}_{Y_\nu}, N_{Y_\nu/X}^n)>\ve\cdot \delta^n$ holds, where $\mathbb{I}_{Y_\nu}$ is the holomorphically trivial line bundle over $Y_\nu$ and $d$ is the distance of ${\rm Pic}^0(Y_\nu)$ induced from the Euclidean norm of the universal covering space $H^1(Y_\nu, \mathcal{O}_{Y_\nu})$ of ${\rm Pic}^0(Y_\nu)$. 
\end{theorem}
Note that $H^1(Y_\nu, \mathcal{O}_{Y_\nu})\cong \mathbb{C}$ holds for each $\nu$, 
since it follows from the adjunction formula that $Y_\nu$'s are elliptic curves. 
Note also that, by using the terms in \cite[\S 2.2]{KT2}, one can reword the sufficient and necessary condition for $(i)$ in Theorem \ref{thm:main} as follows: $(i)$ holds if and only if $N_{Y_\nu/X}$ is an {\it asymptotically positive} element of ${\rm Pic}^0(Y_\nu)$ for each $\nu$. 

For example, one can apply Theorem \ref{thm:main} to the following: 
\begin{example}\label{eg:blP2at9pts_sp_by_KoTLS}
Fix a smooth cubic $C\subset \mathbb{P}^2$, generators $\gamma_1$ and $\gamma_2$ of the fundamental group $\pi_1(C, *)$ of $C$, and eight points $p_1, p_2, \dots, p_8\in C$, where $\mathbb{P}^2$ is the projective plane. 
Then, as is simply observed, the map $I\colon C\to {\rm Pic}^0(C)$ defined by $I(q) := [C]|_C\otimes [-p_1-p_2-\cdots-p_8-q]$ is bijective, where $[D]$ denotes the line bundle which corresponds to the divisor $D$. By identifying ${\rm Pic}^0(C)$ with the group of all the unitary flat line bundles over $C$ (see \cite[\S 1]{U} for example), also consider the map $F\colon \mathbb{R}\to {\rm Pic}^0(C)$ defined by letting $F(\theta)$ be the unitary flat line bundle over $C$ whose monodromy $\rho_{F(\theta)}\colon \pi_1(C, *)\to {\rm U}(1)$ satisfies $\rho_{F(\theta)}(\gamma_1)=1$ and $\rho_{F(\theta)}(\gamma_2)=\exp(2\pi\sqrt{-1}\theta)$ for each $\theta\in \mathbb{R}$, where ${\rm U}(1):=\{t\in\mathbb{C}\mid |t|=1\}$. 
Denoting by $q_\theta$ the point of $C$ which satisfies $I(q_\theta)=F(\theta)$, 
let $\pi_\theta\colon X_\theta \to \mathbb{P}^2$ be the blow-up of $\mathbb{P}^2$ at nine points $p_1, p_2, \dots, p_8$, and $q_\theta$. 
Then the strict transform $Y_\theta := (\pi_\theta^{-1})_*C$ is a non-singular hypersurface of $X$ such that $[Y_\theta]=K_{X_\theta}^{-1}$ and $N_{Y_\theta/X_\theta}\cong (\pi_\theta|_{Y_\theta})^*F(\theta)$. 
Note that $K_{X_\theta}^{-1}$ is semi-positive for each $\theta\in \mathbb{R}$ \cite[Theorem 1.3]{KoTLS}. 
\end{example}
We say that an irrational number $\theta\in \mathbb{R}$ is {\it asymptotically positive} 
if there exists positive numbers $\ve$ and $\delta$ such that, for any positive integer $n$, $\min_{m\in \mathbb{Z}}|n\theta -m|>\ve\cdot \delta^n$ holds. Otherwise, we say that $\theta$ is {\it asymptotically zero}. 
Note that, by using the terms in \cite[\S 2.2]{KT2}, $F(\theta)$ as in Example \ref{eg:blP2at9pts_sp_by_KoTLS} is torsion if $\theta$ is rational, 
asymptotically positive if $\theta$ is asymptotically positive, and 
asymptotically zero if $\theta$ is asymptotically zero. 
Then we have the following: 
\begin{corollary}\label{cor:main}
Let $\theta$ be a real number and $(X_\theta, Y_\theta)$ be as in Example \ref{eg:blP2at9pts_sp_by_KoTLS}. Then the following holds for $M_\theta:=X_\theta\setminus Y_\theta$: \\
$(i)$ When $\theta$ is a rational number, $H^1(M_\theta, \mathcal{O}_{M_\theta})$ is of infinite dimension and Hausdorff. 
$(ii)$ When $\theta$ is an asymptotically positive irrational number, $H^1(M_\theta, \mathcal{O}_{M_\theta})=0$. \\
$(iii)$ When $\theta$ is an asymptotically zero irrational number, $H^1(M_\theta, \mathcal{O}_{M_\theta})$ is of non-Hausdorff type. 
\end{corollary}

See \S \ref{section:ex_of_blp2at9pts} for the details of Example \ref{eg:blP2at9pts_sp_by_KoTLS}, in which the $\ddbar$-problem (i.e. whether the $\ddbar$-lemma holds) on $M_\theta$ is also investigated (Proposition \ref{prop:ddbar_lem_for_Bl9ptsP2}). 
Note that the $\ddbar$-problem on toroidal groups is investigated in \cite{KT1}. See also \cite{KT2}. 

In the proof of Theorem \ref{thm:main}, we will use tubular neighborhoods $V_\nu$ of $Y_\nu$ such that $[Y_\nu]|_{V_\nu}$ is unitary flat for each $\nu$, whose existence is assured by the semi-positivity assumption of $K_X^{-1}=[Y_1+Y_2+\cdots+Y_N]$ and \cite{Ko2020}, \cite[Remark 5.2]{O} (See also \cite{Ko}). 
By considering the Mayer--Vietoris sequence for the open covering $\{M, V_1, V_2, \dots, V_N\}$ of $X$, the proof is reduced to studying $H^1(V_\nu, \mathcal{O}_{V_\nu})$, $H^1(V_\nu^*, \mathcal{O}_{V_\nu^*})$, and the natural map between them, where $V_\nu^*:=V_\nu\cap M (=V_\nu\setminus Y_\nu)$. 
We will investigate them by focusing on the sheaf of holomorphic functions which are constant along the leaves of a holomorphic foliation on $V_\nu$. 

The organization of the paper is as follows. 
\S 2 is a preliminary section. 
In \S \ref{section:prelim_nbhd_and_compl_of_Y}, we will first study the cases where $N_{Y_\nu/X}$ is positive or torsion for some $\nu$. 
In accordance with the strategy of the proof of Theorem \ref{thm:main} as mentioned above, 
a suitable neighborhood $V$ of an elliptic curve $Y$ will be investigated in most part of this paper. 
The precise configuration and notation will be fixed in \S \ref{section:kigou_settei}. 
In \S \ref{section:ueda_lemma_const}, some fundamental facts concerning on Ueda-type estimate for the solution of the \v{C}ech coboudary map, which plays an essential role for connecting irrational number theoretical properties of $N_{Y/X}$ and the complex analytical properties of $V$, are collected. 
\S 3, 4, and 5 will be devoted to the study of $V$: 
Complex analytical properties of $V$ will be investigated mainly by using the assumption that $[Y]$ is an anticanonical divisor in \S 3, 
by applying the dynamical properties of the holomorphic foliation of $V$ in \S 4, 
and by combining the results in \S 3 and 4 in \S 5. 
As an application of the results in \S 5, we show Theorem \ref{thm:main} in \S 6. 
\S 7 is the section for some examples: we will observe toroidal groups of dimension $2$ in \S 7.1, and Example \ref{eg:blP2at9pts_sp_by_KoTLS} in \S 7.2. 

\vskip3mm
{\bf Acknowledgement.} 
The author is grateful to Professor Shigeharu Takayama for his comments on Question \ref{question_hodge_decomp_from_takayamasensei} in \S 7.2. 
He is supported by JSPS KAKENHI Grant Number 23K03119. 

\section{Preliminaries}

\subsection{Fundamental results on the complement of a semi-positive anticanonical divisor}\label{section:prelim_nbhd_and_compl_of_Y}

In this subsection we let $X$ be a connected compact K\"ahler surface such that $K_X^{-1}$ is semi-positive. 
Let $Y$ be a non-singular hypersurface of $X$ such that $D=Y_1+Y_2+\cdots +Y_N$ is an anticanonical divisor, where $\{Y_\nu\}_{\nu=1}^N$ is the set of all the connected components of $Y$. 
First let us show the following: 
\begin{lemma}\label{lem:first_lem_in_prelim}
For a connected compact K\"ahler surface $X$ and a hypersurface $Y=Y_1\cup Y_2\cup \cdots\cup Y_N$ as above, the following holds: \\
$(i)$ Each $Y_\nu$ is an elliptic curve with ${\rm deg}\,N_{Y_\nu/X}\geq 0$. \\
$(ii)$ It holds that 
\[
{\rm dim}\,H^q(X, \mathcal{O}_X)=\begin{cases}
1 & {\rm if}\ q=0\\
1-\frac{1}{12}\left(e(X)+\sum_{\nu=1}^N{\rm deg}\,N_{Y_\nu/X}\right) & {\rm if}\ q=1\\
0 & {\rm if}\ q\geq 2
\end{cases}, 
\]
where $e(X)$ is the Euler number of $X$. \\
$(iii)$ $H^q(M, \mathcal{O}_M)=0$ holds for any $q\geq 2$, where $M=X\setminus Y$. 
\end{lemma}

\begin{proof}
As ${\rm deg}\,K_{Y_\nu}={\rm deg}\,(K_X\otimes [Y_\nu])|_{Y_\nu}=0$ follows from the adjunction formula, $(i)$ holds (As is mentioned in \S 1, the semi-positivity of the normal bundle directly follows from the assumption that $[Y_1+Y_2+\cdots+Y_N]$ is semi-positive). 
As $X$ is compact, $H^0(X, \mathcal{O}_X)=\mathbb{C}$. By the Serre duality, it holds that 
\[
{\rm dim}\,H^2(X, \mathcal{O}_X)
={\rm dim}\,H^{2, 2}(X, [D])
={\rm dim}\,H^0(X, \mathcal{O}_X([-D]))=0
\]
since $D=Y_1+Y_2+\cdots +Y_N$ is an effective anticanonical divisor. 
Therefore $(ii)$ follows from Noether's formula. 
The assertion $(iii)$ follows from \cite[p. 419 (4.1)]{D}, since $M$ is strongly $2$-complete by Ohsawa's theorem \cite{O1984} (See also \cite[p. 417 (3.6)]{D}). 
\end{proof}

\subsubsection{The case where ${\rm deg}\,N_{Y_\nu/X}>0$ for some $\nu$}\label{section:prelim_case_positive}

First let us consider the case where ${\rm deg}\,N_{Y_\nu/X}>0$ for some $\nu$. 
Without loss of generality, we may assume that ${\rm deg}\,N_{Y_1/X}>0$. 
In this case, one has the following: 
\begin{lemma}\label{lemma:firstone_in_prelim_case_positive}
When ${\rm deg}\,N_{Y_1/X}>0$, it holds that $N=1$ (i.e. $Y=Y_1$ is connected) and that the complement $M=X\setminus Y_1$ is strongly $1$-convex. 
Especially $H^1(M, \mathcal{O}_M)$ is of finite dimension. 
\end{lemma}

\begin{proof}
Assume that $N\geq 2$. Then, as $Y_1\cap Y_2=\emptyset$, it follows from Lemma \ref{lem:first_lem_in_prelim} $(i)$ and Hodge index theorem that the first Chern class $c_1([Y_2])$ of $[Y_2]$ is zero. Thus $\int_{Y_2}\omega=0$ holds for any K\"ahler form $\omega$ on $X$, which leads to the contradiction. 

Now the proof is reduced to showing that $M$ is strongly $1$-convex, 
since then the finite-dimensionality of $H^1(M, \mathcal{O}_M)$ follows from 
Andreotti--Grauert's theorem \cite{AG} (See also \cite[p. 425 (4.10)]{D}). 
We will show this by constructing a $C^\infty$ psh (plurisubharmonic) exhaustion function $\psi\colon M\to \mathbb{R}$ such that $\psi|_{V_1\setminus Y_1}$ is spsh (strictly plurisubharmonic) for a neighborhood $V_1$ of $Y_1$. 
By \cite[Proposition (2.2)]{S}, there exists a tubular neighborhood $V_1'$ of $Y_1$ in $X$ and a $C^\infty$ spsh function $\vp\colon V_1'\setminus Y_1\to \mathbb{R}$ such that $\vp(x)\to \infty$ as a point $x\in V_1'\setminus Y_1$ approaches to $Y_1$. 
As the boundary $\del V_1'$ of $V_1'$ is compact, we may assume (by shrinking $V_1'$ if necessary) that $\vp<M$ holds on a neighborhood of $\del V_1'$ for some $M>0$. 
Take a $C^\infty$ non-decreasing convex function $\chi\colon \mathbb{R}\to \mathbb{R}$ such that $\chi(t)=t$ for $t>>1$ and $\chi|_{(-\infty, M]}\equiv c$ holds for some constant $c$. Then the function $\psi\colon M\to \mathbb{R}$ defined by 
\[
\psi(x) = \begin{cases}
\chi\circ\vp(x) & \text{if}\ x\in V_1'\setminus Y_1\\
c & \text{if}\ x\in M\setminus V_1'
\end{cases}
\]
enjoys the required condition. 
\end{proof}

\subsubsection{The case where $N_{Y_\nu/X}$ is torsion for some $\nu$}\label{section:prelim_case_torsion_fibrartion}

Next let us consider the case where $N_{Y_\nu/X}$ is a torsion element of ${\rm Pic}^0(Y_\nu)$ for some $\nu$. 
In this case, by using \cite[Theorem 1.1 (i)]{Ko2020}, we have the following: 
\begin{proposition}\label{prop:existence_of_fibration_f_when_N_tor}
When $N_{Y_\nu/X}$ is a torsion element of ${\rm Pic}^0(Y_\nu)$ for some $\nu$, there exists an elliptic fibration $f\colon X\to B$ onto a compact Riemann surface $B$ such that each $Y_\mu$ is a fiber of $f$. 
Especially, $N_{Y_\mu/X}$ is a torsion element of ${\rm Pic}^0(Y_\mu)$ for any $\mu\in\{1, 2, \dots, N\}$ in this case. 
\end{proposition}

\begin{proof}
Without loss of generality, we may assume that $N_{Y_1/X}$ is torsion. 
In this case, it follows from Lemma \ref{lem:semipositivity_of_each_Y_nu} below and \cite[Theorem 1.1 $(i)$]{Ko2020} that there exists a proper holomorphic surjection $f\colon X\to B$ onto a compact Riemann surface $B$ such that $Y_1$ is a fiber of $f$. 
Note that we may assume $f$ has connected fibers by considering Stein factorization. 
As we are assuming that $X$ is connected, $B$ is also connected. Thus, for each $\mu$, the image $f(Y_\mu)$ is either a point or $B$. If $f(Y_\mu)=B$ for some $\mu>1$, then it holds that $Y_1\cap Y_\mu\not=\emptyset$, which leads to the contradiction. Therefore $Y_\mu$ is a fiber of $f$. 
Denote by $m_\mu$ the multiplicity of the fiber $Y_\mu$: i.e. $m_\mu$ is a positive integer such that $f^*\{f(Y_\mu)\}=m_\mu Y_\mu$ holds as divisors. 
Then it follows from a standard argument that $N_{Y_\mu/X}^{m_\mu}$ is holomorphically trivial (see also \cite[p. 189]{S}). 
For a general fiber $F\subset X$ of $f$, it follows from the adjunction formula that $K_F$ is holomorphically trivial, since both $N_{F/X}$ and $K_X|_F$ is holomorphically trivial. Therefore $F$ is an elliptic curve, from which the assertion follows. 
\end{proof}

\begin{lemma}\label{lem:semipositivity_of_each_Y_nu}
Let $X'$ be a compact complex manifold and $Y'$ be a compact non-singular hypersurface of $X'$. 
Assume that $[Y_1'+Y_2'+\cdots+Y_{N'}']$ is semi-positive, 
where $\{Y_\nu'\}_{\nu=1}^{N'}$ is the set of all the connected components of $Y'$. 
Then $[Y_\nu']$ is semi-positive and $X'\setminus Y_\nu'$ is weakly $1$-complete for each $\nu\in\{1, 2, \dots, N'\}$. 
\end{lemma}

\begin{proof}
It is sufficient to show that $[Y_1']$ is semi-positive and that $X'\setminus Y_1'$ is weakly $1$-complete under the assumption in the lemma. 
Take a $C^\infty$ Hermitian metric $h$ on $[Y_1'+Y_2'+\cdots+Y_{N'}']$ with semi-positive curvature and consider the function $\vp:=-\log |\sigma|_h^2$, where $\sigma\in H^0(X', [Y_1'+Y_2'+\cdots+Y_{N'}'])$ is the canonical section. 
Take a small neighborhood $V_1'$ of $Y_1'$ in $X'$ such that $V_1'\cap Y_\nu=\emptyset$ for any $\nu\geq 2$. 
Then, by modifying $\vp|_{V_1'}$ in the same manner as in the proof of Lemma \ref{lemma:firstone_in_prelim_case_positive}, one obtains a $C^\infty$ psh function $\psi\colon X'\setminus Y_1'\to \mathbb{R}$ such that $\psi|_{W\setminus Y_1'}=\vp|_{W\setminus Y_1'}$ holds for a neighborhood $W$ of $Y_1'$ in $V_1'$. 
As it is a exhaustion function, $X'\setminus Y_1'$ is weakly $1$-complete. 
Define a $C^\infty$ Hermitian metric $h_1$ on $[Y_1']$ by the equation $|\sigma_1|_{h_1}^2=\exp(-\psi)$, where $\sigma_1\in H^0(X', [Y_1'])$ is the canonical section. 
Then, as the local weight function of $h_1$ is psh on a neighborhood of each point of $X'$, the assertion follows. 
\end{proof}

In the rest of this subsection, let us show the following proposition on the cohomology of $M$ in this case. 

\begin{proposition}\label{prop:main_for_prelim_tor_case}
When $N_{Y_\nu/X}$ is a torsion element of ${\rm Pic}^0(Y_\nu)$ for some $\nu$, the following holds: \\
$(i)$ $H^1(M, \mathcal{O}_M)$ is of infinite dimension. \\
$(ii)$ The closure $K$ of $\{0\}$ in $H^1(M, \mathcal{O}_M)$ satisfies ${\rm dim}\,K\leq {\rm dim}\,H^1(X, \mathcal{O}_X)$. Especially, $H^1(M, \mathcal{O}_M)$ is Hausdorff if $H^1(X, \mathcal{O}_X)=0$. 
\end{proposition}

\begin{proof}
Let $f\colon X\to B$ be the fibration as in Proposition \ref{prop:existence_of_fibration_f_when_N_tor}. 
Denote by $\{p_\nu\}$ the image of $Y_\nu$ by $f$ for each $\nu$, and by $B^*$ the complement $B\setminus\{p_1, p_2, \cdots p_N\}$. 
Note that $M=f^{-1}(B^*)$. 
As the assertion $(i)$ follows from Lemma \ref{lem:tor_case_H^1_spectral_seq} $(i)$ below and the fact that $H^1(F, \mathcal{O}_F)\cong \mathbb{C}$ holds for a general fiber $F$ of $f$, we will show the assertion $(ii)$ in what follows. 

Take a small coordinate open ball $\Delta_\nu\subset B$ centered at $p_\nu$ and set $V_\nu:=f^{-1}(\Delta_\nu)$ for each $\nu$. 
Then it follows from Lemma \ref{lem:first_lem_in_prelim} $(ii)$ and the Mayer--Vietoris sequence for the open covering $\{M, V_1, V_2, \dots, V_N\}$ of $X$ that there exists an exact sequence 
\[
\xymatrix{
H^1(X, \mathcal{O}_X) \ar[r]^{\alpha\ \ \ \ \ \ \ \ \ \ \ \ \ \ }&
H^1(M, \mathcal{O}_M)\oplus\bigoplus_{\nu=1}^N H^1(V_\nu, \mathcal{O}_{V_\nu}) \ar[r]^{\ \ \ \ \ \ \ \ \beta}&
\bigoplus_{\nu=1}^N H^1(V_\nu^*, \mathcal{O}_{V_\nu^*}) \ar[r]&
0, 
}
\]
where $\Delta_\nu^*:=\Delta_\nu\setminus\{p_\nu\}$ and $V_\nu^*:=f^{-1}(\Delta_\nu^*)$ for each $\nu$. 
As the morphism $\beta$ is continuous by construction and $\textstyle\bigoplus_{\nu=1}^N H^1(V_\nu^*, \mathcal{O}_{V_\nu^*})$ is Hausdorff by Lemma \ref{lem:tor_case_H^1_spectral_seq} $(ii)$, the kernel ${\rm Ker}\,\beta$ of $\beta$ is a closed subset which contains the origin. 
Thus one has that $K\oplus K'$ is included in ${\rm Ker}\,\beta\ (={\rm Image}\,\alpha)$, where $K'$ is the closure of $\{0\}$ in $\textstyle\bigoplus_{\nu=1}^N H^1(V_\nu, \mathcal{O}_{V_\nu})$. 
Thus it holds that ${\rm dim}\,K\leq {\rm dim}({\rm Image}\,\alpha) \leq {\rm dim}\,H^1(X, \mathcal{O}_X)$. 
Especially, if $H^1(X, \mathcal{O}_X)=0$, then it holds that $K=\{0\}$, which means that $\{0\}$ is a closed subset of $H^1(M, \mathcal{O}_M)$. As is well-known, $H^1(M, \mathcal{O}_M)$ is Hausdorff in this case (see \cite[I 2.3]{SW} for example), from which the assertion follows. 
\end{proof}

\begin{lemma}\label{lem:tor_case_H^1_spectral_seq}
Let $f\colon X\to B$ be as in the proof of Proposition \ref{prop:main_for_prelim_tor_case} 
and $D\subset B$ be a Stein domain. Then, for the preimage $W:=f^{-1}(D)$, the following holds: \\
$(i)$ For any finitely many regular values $q_1, q_2, \dots, q_\ell\in D$ of $f|_W$, the restriction map
\[
H^1(W, \mathcal{O}_W) \to \bigoplus_{\lambda=1}^\ell H^1(f^{-1}(q_\lambda), \mathcal{O}_{f^{-1}(q_\lambda)})
\]
is surjective. \\
$(ii)$ 
If $H^2(D, \mathbb{Z})=0$ and 
$f|_W$ is a submersion, then there exists an element $\alpha \in H^1(W, \mathcal{O}_{W})$ such that the map 
\[
H^0(W, \mathcal{O}_{W})\ni \eta\mapsto \eta\cdot \alpha\in H^1(W, \mathcal{O}_{W})
\]
is an isomorphism between topological vector spaces, where $H^0(W, \mathcal{O}_{W})$ is regarded as a Fr\'echet space by using the topology of uniform convergence on compact sets. 
\end{lemma}
%

\begin{proof}
Consider the exact sequence 
\[
H^1(D, f_*\mathcal{O}_W) \to H^1(W, \mathcal{O}_W) \to H^0(D, R^1f_*\mathcal{O}_W)\to H^2(D, f_*\mathcal{O}_W)
\]
obtained by considering the five-term exact sequence for the Leray spectral sequence. 
As $D$ is Stein and $f$ is proper, it follows from Grauert's direct image theorem and Oka--Cartan vanishing theorem that $H^1(D, f_*\mathcal{O}_W)=0$ and $H^2(D, f_*\mathcal{O}_W)=0$ (See \cite[Chapter I\!I\!I]{GPR} for example). 
Therefore the map $H^1(W, \mathcal{O}_W) \to H^0(D, R^1f_*\mathcal{O}_W)$ is a linear bijection. 

In order to show the assertion $(i)$, take finitely many regular values $q_1, q_2, \dots, q_\ell\in D$ of $f|_W$. 
Let $I$ be the defining ideal sheaf of the analytic subset $\{q_1, q_2, \dots, q_\ell\}$ of $D$. 
Then the short exact sequence $0\to I\cdot R^1f_*\mathcal{O}_W\to R^1f_*\mathcal{O}_W\to R^1f_*\mathcal{O}_W/(I\cdot R^1f_*\mathcal{O}_W)\to 0$ induces the exact sequence 
\[
H^0(D, R^1f_*\mathcal{O}_W)\to H^0(D, R^1f_*\mathcal{O}_W/(I\cdot R^1f_*\mathcal{O}_W)) \to 0, 
\]
since $H^1(D, I\cdot R^1f_*\mathcal{O}_W)=0$ holds again by Grauert's direct image theorem and Oka--Cartan vanishing theorem, from which the assertion $(i)$ follows. 

In what follows we will show the assertion $(ii)$. 
As $f|_W$ is a submersion and $f$ is an elliptic fibration, one can identify $R^1f_*\mathcal{O}_W$ with the sheaf $\mathcal{O}_D(L)$ of the holomorphic sections of a line bundle $L$ on $D$ (See \cite[I\!I\!I \S 4 Theorem 4.7 (d)]{GPR} for example). 
As $c_1(L)\in H^2(D, \mathbb{Z})=0$ and $H^1(D, \mathcal{O}_D)=0$ by the Steinness of $D$, it follows from the long exact sequence induced by the exponential sequence that $L$ is holomorphically trivial. 
Therefore, by considering a global trivialization of $L$, one obtains an element $\alpha \in H^1(W, \mathcal{O}_W)\ (\cong H^0(D, R^1f_*\mathcal{O}_W))$ such that the map 
\begin{equation}\label{map:timesapha_whichisisom}
H^0(D, \mathcal{O}_{D})\ni \eta\mapsto (f^*\eta)\cdot \alpha\in H^1(W, \mathcal{O}_{W})
\end{equation}
is a linear bijection. 
As $H^0(W, \mathcal{O}_{W})\cong H^0(D, \mathcal{O}_D)$ holds by the properness of $f$, it is sufficient to show that this map is a homeomorphism. 

As a preparation, first let us show that 
\begin{equation}\label{equation:H0omegacapImage=0}
(H^0(W, \mathcal{O}_W)\cdot \omega)\cap \overline{B^{0, 1}(W)}=0
\end{equation}
holds as subspaces of $Z^{0, 1}(W)$, where $\omega\in Z^{0, 1}(W)$ is a representative of the class $\alpha$. 
Take an element $\eta\in H^0(D, \mathcal{O}_D)$ such that $(f^*\eta)\cdot \omega\in \overline{B^{0, 1}(W)}$, and a point $p\in D$. 
As the restriction map $\xi\colon Z^{0, 1}(W)\to H^1(f^{-1}(p), \mathcal{O}_{f^{-1}(p)})$ is continuous and $\xi(B^{0, 1}(W))=\{0\}\subset H^1(f^{-1}(p), \mathcal{O}_{f^{-1}(p)})$, it holds that $\eta(p)\cdot (\alpha|_{f^{-1}(p)})=\xi((f^*\eta)\cdot \omega)=0$. Therefore, by the construction of $\alpha$, it holds that $\eta(p)=0$, from which (\ref{equation:H0omegacapImage=0}) follows. 

Take an element $a\in \overline{B^{0, 1}(W)}$. 
Then, as the map (\ref{map:timesapha_whichisisom}) is bijective, one can take an element $\zeta\in H^0(W, \mathcal{O}_W)$ such that $\zeta\omega-a\in B^{0, 1}(W)$. 
Thus we have that $\zeta\omega\in \overline{B^{0, 1}(W)}$, from which and (\ref{equation:H0omegacapImage=0}) it follows that $\zeta\omega=0$. 
Therefore it holds that $a\in B^{0, 1}(W)$, from which it follows that $B^{0, 1}(W)$ is a closed subset of the Fr\'echet space $Z^{0, 1}(W)$. 
Thus the quotient space $H^1(W, \mathcal{O}_W)$ is also a Fr\'echet space. 
As the map (\ref{map:timesapha_whichisisom}) is a continuous linear bijection between Fr\'echet spaces, it follows from the open mapping theorem \cite[2.11]{R} that the map (\ref{map:timesapha_whichisisom}) is a homeomorphism.
\end{proof}

\subsubsection{Preliminary observation for the case where $N_{Y_\nu/X}$ is non-torsion for any $\nu$}\label{section:prelim_case_non_tor_1compl_etc}

In what follows, we mainly investigate the case where $N_{Y_\nu/X}$ is a non-torsion element of ${\rm Pic}^0(Y_\nu)$ for any $\nu\in \{1, 2, \dots, N\}$, since otherwise we know the complex analytic structure of $M$ quite well, as we observed in 
\S \ref{section:prelim_case_positive} and \ref{section:prelim_case_torsion_fibrartion}. 
Note that the K\"ahlerness assumption for $X$ is not needed in all the arguments from this subsection. 

\begin{proposition}\label{prop:existence_of_Vw_cptsurf}
Let $X$ be a non-singular connected compact complex surface such that $K_X^{-1}$ is semi-positive. 
Let $Y$ be a non-singular hypersurface of $X$ such that $D=Y_1+Y_2+\cdots +Y_N$ is an anticanonical divisor, where $\{Y_\nu\}_{\nu=1}^N$ is the set of all the connected components of $Y$. Assume that $N_{Y_\nu/X}$ is a non-torsion element of ${\rm Pic}^0(Y_\nu)$ for any $\nu\in \{1, 2, \dots, N\}$. 
Then, for each $\nu$, there exist an open neighborhood $V_\nu\subset X$ of $Y_\nu$, 
a finite open covering $\{V_{\nu, j}\}$ of $V_\nu$, 
and a holomorphic defining function $w_{\nu, j}$ of $V_{\nu, j}\cap Y_\nu$ on $V_{\nu, j}$ such that $w_{\nu, j}=t_{\nu, jk}w_{\nu, k}$ holds for some $t_{\nu, jk}\in {\rm U}(1)$ on each $V_{\nu, jk}:=V_{\nu, j}\cap V_{\nu, k}$. 
Especially, each $V_\nu^*:=V_\nu\setminus Y_\nu$ admits a family $\{H_{\nu, t}\}_{t\in (0, 1)}$ of compact Levi-flat hypersurfaces such that each leaf of $H_{\nu, t}$ is dense in $H_{\nu, t}$ for any $t\in (0, 1)$. 
Moreover, the complement $M:=X\setminus Y$ is weakly $1$-complete, whereas it is not strongly $1$-convex. 
\end{proposition}

\begin{proof}
By Lemma \ref{lem:semipositivity_of_each_Y_nu}, 
$M$ is weakly $1$-complete and $[Y_\nu]$ is semi-positive. 
Thus the existence of $\{(V_{\nu, j}, w_{\nu, j})\}$ as stated is a consequence from \cite{Ko2020}, \cite[Remark 5.2]{O} (see Remark \ref{rmk:nonKa_case} for the details). 
For a sufficiently small positive number $\ve>0$ and for each $t\in (0, 1)$, 
set $H_{\nu, t}:=\bigcup_j\{p\in V_{\nu, j}\mid |w_{\nu, j}(p)|=\ve\cdot t\}$. 
Then $H_{\nu, t}$ is a Levi-flat hypersurface, whose leaves are locally defined by $d w_{\nu, j}=0$. 
As the holonormy of this foliation coincides with the monodromy of the unitary flat line bundle $N_{Y_\nu/X}$, it follows from the assumption that each leaf is dense in $H_{\nu, t}$. 
Thus the assertion follows from Lemma \ref{lem:eachpshonHisconst} below. 
\end{proof}

\begin{remark}\label{rmk:nonKa_case}
Let $X$ and $Y_\nu$ be as in Proposition \ref{prop:existence_of_Vw_cptsurf}. 
Here we focus on a neighborhood of a fixed $Y_\nu$ and drop the index $\nu$ for simplicity (No confusion will be caused by this abbreviation, 
since here we will not use any assumption on the relation between $K_X^{-1}$ and $Y:=Y_\nu$, whereas we will use the fact that $Y$ is a non-singular elliptic curve embedded in $X$ whose normal bundle is topologically trivial). 
The existence of $\{(V_j, w_j)\}$ as in the statement of Proposition \ref{prop:existence_of_Vw_cptsurf} is shown when $X$ is K\"ahler \cite[Corollary 1.5]{Ko}. 
In accordance with the proof of \cite[Theorem 1.4 $(b)$]{Ko2020}, we can also describe the outline of the proof of this fact as follows (See also \cite[Observation 6.1]{Ko}): 
\begin{description}
\item[Step 1] For a psh exhaustion function $\psi$ on $X\setminus Y$ as in the proof of Lemma \ref{lem:semipositivity_of_each_Y_nu}, show that the level set $\{\psi=N\}$ is Levi-flat for large $N$. 
\item[Step 2] Show that there exists a holomorphic foliation $\mathcal{F}$ on a neighborhood $V$ of $Y$ whose leaf is either $Y$ or a leaf of the Levi-foliation of $\{\psi=N\}$ for some $N$. 
\item[Step 3] Show that the holonomy of $\mathcal{F}$ along $Y$ is $\mathrm{U}(1)$-rotation, and construct $\{(V_j, w_j)\}$ as in the statement of Proposition \ref{prop:existence_of_Vw_cptsurf} by considering the leafwise constant extension of a suitable coordinate of a transversal. 
\end{description}

Note that we may assume that the complex Hessian of $\psi$ has at least one positive eigenvalue at each point of $V\setminus Y$, where $V$ is a small neighborhood of $Y$ (See \cite[Proposition 2.1]{Ko2020} or \cite[\S 3.1]{Ko}). 
As is explained in \cite[Remark 5.2]{O}, one can also carry out Step 1 when $X$ is a non-K\"ahler smooth surface as follows: 
Assume that $X$ is non-K\"ahler and that $\{\psi=N\}$ is not Levi-flat for sufficiently large $N$. Then the Hermitian metric on $[Y]$ constructed from $\psi$ in the same manner as in the proof of Lemma \ref{lem:semipositivity_of_each_Y_nu} has a semi-positive curvature which is positive on a neighborhood of a point. Therefore it follows from Siu's solution  of Grauert–Riemenschneider conjecture \cite[Theorem 1]{Siu} that $X$ is Moishezon. 
Thus the contradiction occurs, since a non-singular Moishezon surface is projective \cite[Theorem 3.1]{Kod1} (See also \cite[VII Corollary 6.11]{GPR}). 

In the arguments for Step 2 (=Proof of \cite[Lemma 4.2 (ii), (iii)]{Ko2020}), the K\"ahlerness assumption of $X$ is used only for reducing the proof to the case where the restriction map $H^1(X, \mathcal{O}_X)\to H^1(Y, \mathcal{O}_{Y})$ is injective. 
In the arguments for Step 3 (=Proofs of Lemma 4.4 and Proposition 4.5, and the arguments after Proposition 4.5 in \cite{Ko2020}), the K\"ahlerness assumption of $X$ is not used. 
Therefore, in order to prove the existence of $\{(V_j, w_j)\}$ as in the statement of Proposition \ref{prop:existence_of_Vw_cptsurf}, it is sufficient to show the injectivity of $H^1(X, \mathcal{O}_X)\to H^1(Y, \mathcal{O}_{Y})$ when $X$ is neither K\"ahler nor elliptic, since \cite[Corollary 1.5]{Ko} can be applied when $X$ is K\"ahler and any connected component $Y$ of an anticanonical divisor $D$ of an elliptic surface which comes from a non-singular hypersurface is a fiber (Here we say that a surface is elliptic if it admits a proper holomorphic map onto a Riemann surface whose general fiber is an elliptic curve). 
Assume that $X$ is non-K\"ahler non-elliptic smooth surface. 
Then the algebraic dimension $a(X)$ of $X$ is equal to zero, since $X$ is projective if $a(X)=2$ and is elliptic if $a(X)=1$ \cite{Kod1}. 
Let $\pi\colon X\to S$ be the composition of finitely many blow-up morphisms such that $S$ is a minimal surface. 
Note that we may assume $S$ is non-elliptic (since otherwise $X$ is elliptic), and that $a(S)=0$. 
Therefore, by Kodaira's classification \cite[Theorem 55]{Kod2}, we may assume that $S$ is a surface of class VII$_0$ with $a(S)=0$. 
In this case, by Nakamura's theorems, the image $Z:=\pi(Y)$ of $Y$ is a non-singular elliptic curve in $S$ and the restriction map $H^1(S, \mathcal{O}_S)\to H^1(Z, \mathcal{O}_Z)$ is injective \cite[(2.2.2), (2.6)]{I}. 
As $\pi^*\colon H^1(S, \mathcal{O}_S)\to H^1(X, \mathcal{O}_X)$ is isomorphism, the assertion holds. 
\end{remark}

\begin{lemma}\label{lem:eachpshonHisconst}
For $V_\nu^*$ and $\{H_{\nu, t}\}_{t\in (0, 1)}$ as in 
Proposition \ref{prop:existence_of_Vw_cptsurf}, the following holds: 
For any psh function $\psi\colon V_\nu^*\to \mathbb{R}$ and any $t\in (0, 1)$, 
$\psi|_{H_{\nu, t}}$ is constant. 
Especially, there is no spsh function on $V_\nu^*$. 
\end{lemma}

\begin{proof}
Take $t\in (0, 1)$. 
As $\psi$ is upper semi-continuous, the maximum of $\psi|_{H_{\nu, t}}$ is attained at a point $x\in H_{\nu, t}$. Let $\mathcal{L}\subset H_{\nu, t}$ the leaf such that $x\in \mathcal{L}$. 
By the maximum principle, we have that $\psi|_{\mathcal{L}}\equiv C$ holds for a constant $C\in \mathbb{R}$. Thus the assertion follows from the density of $\mathcal{L}$. 
Note that the complex Hessian of $\psi$ is zero along leaves and thus it cannot be spsh. 
\end{proof}

In what follows, instead of $H^1(M, \mathcal{O}_M)$ we often investigate the \v{C}ech cohomology of $M$ under the assumption of Proposition \ref{prop:existence_of_Vw_cptsurf} 
by using an open covering of $M$ as follows. 
By shrinking $V_{\nu}$'s and by changing the scaling of $w_{\nu, j}$, 
we may assume that each $w_{\nu, j}$ is defined on a neighborhood of $\overline{V_{\nu, j}}$ in $X$ for each $j$ and that $V_{\nu}=\textstyle\bigcup_j\{|w_{\nu, j}|< 1\}$ holds. 
Take a finite Stein open covering $\{W_\alpha\}$ of the complement $X\setminus\textstyle\bigcup_{\nu, j}\{|w_{\nu, j}|\leq 1/2\}$ (Consider a finite Stein open covering of a relatively compact subset $X\setminus\textstyle\bigcup_{\nu}\overline{V_\nu}$ together with $\{1/2<|w_{\nu, j}|<1\}$'s, for example). 
We use $\mathcal{W}:=\{V_{\nu, j}^*\}_{\nu, j}\cup \{W_\alpha\}$ as an open covering of $M$. 

By \v{C}ech--Dolbeault correspondence, one can determine the dimension of $H^1(M, \mathcal{O}_M)$ instead by observing the \v{C}ech cohomology $\check{H}^1(\mathcal{W}, \mathcal{O}_M)$. 
One can also check that $H^1(M, \mathcal{O}_M)$ is not Hausdorff instead by observing $\check{H}^1(\mathcal{W}, \mathcal{O}_M)$ by the following: 
\begin{lemma}\label{lem:comparizon_top_nonhausdoeff_CDcorresp}
Let $Z$ be a complex manifold and $\{W_j\}$ be a finite Stein open covering of $Z$. 
Assume that $H^1(\{W_j\}, \mathcal{O}_Z)$ is not Hausdorff by regarding it as the topological vector space obtained by the quotient $\check{Z}^1(\{W_j\}, \mathcal{O}_Z)/\check{B}^1(\{W_j\}, \mathcal{O}_Z)$ (Here we are regarding the space $\check{C}^{q}(\{W_j\}, \mathcal{O}_Z)$ of all the \v{C}ech $q$-cochains as a Fr\'echet space by using the topology of uniform convergence on compact sets for each integer $q$, and topologizing the space $\check{B}^1(\{W_j\}, \mathcal{O}_Z)$ of $1$-coboundaries and $\check{Z}^1(\{W_j\}, \mathcal{O}_Z)$ of $1$-cocycles by the relative topologies as subspaces of $\check{C}^{1}(\{W_j\}, \mathcal{O}_Z)$). 
Then $H^1(Z, \mathcal{O}_Z)$ is not Hausdorff. 
\end{lemma}

\begin{proof}
Fix a partition of unity $\{(W_j, \rho_j)\}$ subordinate to $\{W_j\}$. 
As is well-known, the map $F\colon \check{Z}^1(\{W_j\}, \mathcal{O}_Z)\to A^{0, 1}(Z)$ defined by letting $F(\{(W_{jk},\ f_{jk})\})$ be the $(0, 1)$-form such that
\[
F(\{(W_{jk},\ f_{jk})\})|_{W_j} = -\sum_{\ell \not=j}f_{j\ell}\cdot \delbar\rho_\ell
\]
holds on each $W_j$ induces a linear bijection $\widehat{F}\colon \check{H}^1(\{W_j\}, \mathcal{O}_Z)\to H^1(Z, \mathcal{O}_Z)$ (\v{C}ech--Dolbeault correspondence). 
As $F$ is continuous, $\widehat{F}$ is also continuous, from which the assertion follows. 
\end{proof}

As is explained in \S 1, our argument for proving Theorem \ref{thm:main} is based on 
the Mayer--Vietoris sequence for the open covering $\{M, V_1, V_2, \dots, V_N\}$ of $X$ 
and thus it is essential for us to investigate $H^1(V_\nu, \mathcal{O}_{V_\nu})$, $H^1(V_\nu^*, \mathcal{O}_{V_\nu^*})$, and the restriction map between them for each $\nu$. 

\subsection{Configuration and notations for a neighborhood of an elliptic curve}\label{section:kigou_settei}

Based on the observation in the previous subsection, 
in what follows (\S 3, 4, and 5) we will investigate the \v{C}ech cohomologies of $V$ and $V^*:=V\setminus Y$, and the natural map between them, for an elliptic curve $Y$ embedded in a surface and a suitable neighborhood $V$ of $Y$ which has the same properties as $Y_\nu$ and $V_\nu$ as in Proposition \ref{prop:existence_of_Vw_cptsurf}. 
As a preparation, we will clarify the assumption of $V$ and $Y$ and fix notations concerning on them in this subsection. 

Let $X$ be a non-singular complex surface, $Y$ be a non-singular elliptic curve holomorphically embedded in $X$, 
and $V$ be a neighborhood of $Y$ in $X$. 
Assume that $K_V^{-1}=[Y]$ holds on $V$ (i.e. there exists a meromorphic $2$-form $\eta$ on $V$ whose divisor is $-Y$), $N_{Y/V}$ is a non-torsion element of ${\rm Pic}^0(Y)$, and that there exist a finite covering $\{V_j\}$ of $V$ and holomorphic defining function $w_j$ of $U_j:=V_j\cap Y$ in each $V_j$ such that $w_j=t_{jk}w_k$ holds for some $t_{jk}\in {\rm U}(1)$ on each $V_{jk}:=V_j\cap V_k$. 
Note that $\{(V_j, \log |w_j|)\}$ patches to define a psh function on a neighborhood of $Y$. 
By shrinking $V$, we will assume that $V$ is a sublevel set of this function. 
Additionally, by shrinking $V$, changing the scaling of $w_j$, and taking a refinement of $\{V_j\}$, 
we will always assume the following in what follows: 
each $U_j$ is a coordinate open disc, 
each function $w_j$ is defined on a neighborhood of $\overline{V_j}$ and 
$V_j=\{(z_j, w_j)\mid z_j\in U_j, |w_j|<1\}\cong U_j\times \{w_j\in\mathbb{C}\mid |w_j|<1\}$ holds, 
$\{U_j\}$ is sufficiently fine so that $\check{H}^1(\{U_j\}, {\rm U}(1))\cong H^1(Y, {\rm U}(1))$ ($\cong {\rm Pic}^0(Y)$), and 
$U_{jk}:=U_j\cap U_k$ is empty if and only if so is $V_{jk}$. 
Denote by $V^*$ the complement $V\setminus Y$. 
We will let $V_j^*:=V_j\cap V^*$ and use $\{V_j^*\}$ as an open covering of $V^*$. 

Let $z_j$ be a local coordinate on $U_j$ which comes from the Euclidean coordinate of the universal covering of $Y$. Then one can take a constant $A_{jk}\in \mathbb{C}$ such that 
$z_j=A_{jk}+z_k$ holds on each $U_{jk}$. 
We will take a holomorphic extension of $z_j$ to $V_j$ and use $(z_j, w_j)$ as coodinates of $V_j$. Note that we will show by using the assumption $K_V^{-1}=[Y]$ that there is a suitable manner for extending $z_j$ in \S \ref{section:z_ext_K=-Y}. 
In \S 4 and 5, we will also assume that $z_j$ is extended in such a manner (More precisely, we will assume that $z_j$'s enjoy the property as in Proposition \ref{prop:H0_of_O/CF} $(1)$). 

Denote by $\mathcal{F}$ the holomorphic foliation on $V$ whose leaves are locally defined by $d w_j=0$ on each $V_j$, 
by $\mathcal{C}_{\mathcal{F}}$ the sheaf of holomorphic functions on open subsets of $V$ which are constant along leaves of $\mathcal{F}$. Note that each sections of $\mathcal{C}_{\mathcal{F}}$ can be locally regarded as a holomorphic function with one variable $w_j$ on a neighborhood of each point of $V_j$. 
Similarly, we will denote by $\mathcal{C}_{\mathcal{F}}^*$ the sheaf of holomorphic functions on open subsets of $V^*$ which are constant along leaves of $\mathcal{F}$. 

%
%

Note that, for holomorphic functions on $V$ and $V^*$, we have the following: 
\begin{lemma}\label{lem:nononconstantholfunctionexistsonV}
Any holomorphic function on $V^*$ is a constant function. 
Especially it holds that 
$H^0(V, \mathcal{O}_V)=H^0(V^*, \mathcal{O}_{V^*})\cong \mathbb{C}$. 
\end{lemma}

\begin{proof}
As each levelsets of the psh function $\{(V_j, \log|w_j|)\}$ is a compact Levi-flat hypersurface such that each leaf is dense, the assertion follows from a standard argument as in the proof of Lemma \ref{lem:eachpshonHisconst} (see also \cite[Lemma 2.2]{KU}). 
\end{proof}

We divide the situation in to three cases in accordance with a dynamical or an irrational theoretical property of the normal bundle $N_{Y/V}$ of $Y$ as a point of ${\rm Pic}^0(Y)$ with respect to the distance $d$ on ${\rm Pic}^0(Y)$ which is induced from the Euclidean norm of the universal covering of ${\rm Pic}^0(Y)$: 
\begin{description}
\item[Case I] For any real number $R$ which is larger than $1$, there exists a positive number $M_R$ such that $1/d(\mathbb{I}_Y, N_{Y/X}^n)\leq M_R\cdot R^n$ holds for any positive integer $n$. 
\item[Case I\!I] Not in Case I, and there exist positive numbers $R$ and $M_R$ with $R>1$ such that $1/d(\mathbb{I}_Y, N_{Y/X}^n)\leq M_R\cdot R^n$ holds for any positive integer $n$. 
\item[Case I\!I\!I] Neither in Case I nor I\!I: i.e. for any positive numbers $R$ and $M$ with $R>1$, $1/d(\mathbb{I}_Y, N_{Y/X}^n)> M\cdot R^n$ holds for some positive integer $n$. 
\end{description}
Here $\mathbb{I}_Y$ denotes the holomorphically trivial line bundle on $Y$. 
The conditions in each of the cases are applied to the arguments for investigating $V$ and $V^*$ via Ueda type estimates for the \v{C}ech coboundary map (see the next subsection for the details). 

We say that $V$ is a {\it holomoprhic tubular neighborhood} of $Y$ if there exists a biholomorphism between $V$ and a neighborhood of the zero section in the total space of $N_{Y/V}$ which maps $Y$ to the zero section. 
As is simply observed, $V$ is a holomoprhic tubular neighborhood if and only if $z_j$'s can be extended so that $z_j=A_{jk}+z_k$ holds (not only on $U_{jk}$ but also) on $V_{jk}$ for each $j$ and $k$. 
In Case I and I\!I, we will show that one can shrink $V$ so that it is a holomorphic tubular neighborhood of $Y$ (Proposition \ref{prop:hol_tub_nbhd_aru}). 
In such cases, by using $z_j$'s which satisfy $z_j=A_{jk}+z_k$ on each $V_{jk}$ and by using a suitable $\{V_j\}$ (see Remark \ref{rmk:niceVjs_whenCaseI_II} below), the following clearly holds: For any \v{C}ech $1$-cocycle $\{(V_{jk}, g_{jk})\}\in \check{Z}^1(\{V_j\}, \mathcal{O}_V)$ and any coefficient function $g_{jk, n}\colon U_j\to \mathbb{C}$ of the expansion
\[
g_{jk}(z_j, w_j) = \sum_{n=0}^\infty g_{jk, n}(z_j)\cdot w_j^n, 
\]
it holds that $\{(U_{jk}, g_{jk, n})\}$ is an element of $\check{Z}^1(\{U_j\}, \mathcal{O}_Y(N_{Y/V}^{-n}))$ (See also Remark \ref{rmk:keisikikai_C_F_or_holtubnbhd}). 
By using this fact, we can reduce the study of $\check{H}^1(\{V_j\}, \mathcal{O}_V)$ to the estimate of the operator norms of the coboudary maps $\delta\colon \check{B}^1(\{U_j\}, \mathcal{O}_Y(N_{Y/X}^{-n}))\to \check{Z}^1(\{U_j\}, \mathcal{O}_Y(N_{Y/X}^{-n}))$. 
Even when we can not assure that $V$ is a holomorphic tubular neighbohood of $Y$, the same type arguments can be applied to the study of $\check{H}^1(\{V_j\}, \mathcal{C}_{\mathcal{F}})$, as we will see in \S 4. 
The relation between $\check{H}^1(\{V_j\}, \mathcal{O}_V)$ and $\check{H}^1(\{V_j\}, \mathcal{C}_{\mathcal{F}})$ can be observed by using the long exact sequence 
\[
0 \to H^0(V, \mathcal{C}_{\mathcal{F}}) \to H^0(V, \mathcal{O}_V) \to H^0(V, \mathcal{O}_V/\mathcal{C}_{\mathcal{F}}) 
\to \check{H}^1(\{V_j\}, \mathcal{C}_{\mathcal{F}}) \to H^1(V, \mathcal{O}_V)\to\cdots 
\]
induced from the short exact sequence 
$0 \to \mathcal{C}_{\mathcal{F}} \to \mathcal{O}_V \to \mathcal{O}_V/\mathcal{C}_{\mathcal{F}} \to 0$. 
From 
Lemma \ref{lem:nononconstantholfunctionexistsonV}, 
we obtain an exact sequence
\begin{equation}\label{eq:main_exact_seq}
0 \to H^0(V, \mathcal{O}_V/\mathcal{C}_{\mathcal{F}}) 
\to \check{H}^1(\{V_j\}, \mathcal{C}_{\mathcal{F}}) \to H^1(V, \mathcal{O}_V). 
\end{equation}
Similarly we obtain an exact sequence
\begin{equation}\label{eq:main_exact_seq_*version}
0 \to H^0(V^*, \mathcal{O}_{V^*}/\mathcal{C}_{\mathcal{F}}^*) 
\to \check{H}^1(\{V^*_j\}, \mathcal{C}_{\mathcal{F}}^*) \to H^1(V^*, \mathcal{O}_{V^*}). 
\end{equation}
As will be seen in \S \ref{section:z_ext_K=-Y}, we can determine $H^0(V, \mathcal{O}_V/\mathcal{C}_{\mathcal{F}})$ and $H^0(V^*, \mathcal{O}_{V^*}/\mathcal{C}_{\mathcal{F}}^*)$. 

\begin{remark}\label{rmk:niceVjs_whenCaseI_II}
As is mentioned above, we will show that one can shrink $V$ so that it is a holomorphic tubular neighborhood of $Y$ in Case I and I\!I (Proposition \ref{prop:hol_tub_nbhd_aru}). 
After it is verified, we alway use the following covering and coordinates whenever we work in Case I and I\!I: 
For a sufficiently fine open covering $\{U_j\}$ of $Y$ and coordinates $z_j$'s which satisfies $z_j=A_{jk}+z_k$ for $A_{jk}\in\mathbb{C}$ on each $U_{jk}$, we let $V_j:=p^{-1}(U_j)$ and use the extension $z_j\circ p$ of $z_j$, where $p\colon V\to Y$ is the holomorphic retraction which corresponds to the projection $N_{Y/V}\to Y$. 
Note that, by denoting also by $z_j$ the extension $z_j\circ p$ and by using $(z_j, w_j)$ as coordinates on $V_j$, it is clear that $V_j\cong U_j\times \{w_j\in \mathbb{C}\mid |w_j|<1\}$, $V_{jk}\cong U_{jk}\times \{w_j\in \mathbb{C}\mid |w_j|<1\}$. 
\end{remark}

\subsection{Ueda type estimates}\label{section:ueda_lemma_const}
Let $Y$ be an smooth elliptic curve and $\{U_j\}$ be a finite open covering as in the previous subsection. 
For each $F\in {\rm Pic}^0(Y)\setminus\{\mathbb{I}_Y\}$, denote by $K(F)$ he constant
\[
\sup\left\{\left.\frac{\max_j \sup_{U_j}|f_j|}{\max_{j, k}\sup_{U_{jk}}|f_j-s_{jk}\cdot f_k|}\,\right|\,\{(U_j, f_j)\}\in\check{C}^0(\{U_j\}, \mathcal{O}_Y(F))\setminus\{0\},\ \max_j \sup_{U_j}|f_j|<\infty\right\}, 
\]
where $s_{jk}$'s are elements of ${\rm U}(1)$ such that $[\{(U_{jk}, s_{jk})\}]=F$ holds as elements of $\check{H}^1(\{U_j\}, {\rm U}(1))$. 
Note that, as is simply observed, the definition of $K(F)$ does not depend on the choice of such $s_{jk}$'s. 
Note also that the denominator $\max_{j, k}|f_j-s_{jk}\cdot f_k|$ is not zero for any $\{(U_j, f_j)\}\in\check{C}^0(\{U_j\}, \mathcal{O}_Y(F))\setminus\{0\}$, since $H^0(Y, \mathcal{O}_Y(F))=0$ (See \cite[Lemma 2.3]{HK} for example). 
Similarly, 
\[
K_{\rm const}(F) := \sup\left\{\left.\frac{\max_j |f_j|}{\max_{j, k}|f_j-s_{jk}\cdot f_k|}\,\right|\,\{(U_j, f_j)\}\in\check{C}^0(\{U_j\}, \mathbb{C}(F))\setminus\{0\}\right\}
\]
is also well-defined for each $F= [\{(U_{jk}, s_{jk})\}]\in \check{H}^1(\{U_j\}, {\rm U}(1))\setminus\{\mathbb{I}_Y\}$, where $\mathbb{C}(F)$ denotes the sheaf of locally constant sections of $F$ with respect to the unitary flat structure of $F$. 

For these constants, the following holds:
\begin{lemma}\label{lem:ueda_type_estim_K_Kconst_explicitly}
For any $F\in {\rm Pic}^0(Y)\setminus\{\mathbb{I}_Y\}$, the following holds: \\
$(i)$ For any $\{(U_{jk}, g_{jk})\}\in\check{B}^1(\{U_j\}, \mathcal{O}_Y(F))$, there uniquely exists an element $\{(U_j, f_j)\}\in\check{C}^0(\{U_j\}, \mathcal{O}_Y(F))$ such that $\delta(\{(U_j, f_j)\})=\{(U_{jk}, g_{jk})\}$ holds. Moreover, 
\[
\max_j \sup_{U_j}|f_j|\leq K(F)\cdot \max_{j, k}\sup_{U_{jk}}|g_{jk}|
\]
holds. \\
$(ii)$ For any $\{(U_{jk}, g_{jk})\}\in\check{B}^1(\{U_j\}, \mathbb{C}(F))$, there uniquely exists an element $\{(U_j, f_j)\}\in\check{C}^0(\{U_j\}, \mathbb{C}(F))$ such that $\delta(\{(U_j, f_j)\})=\{(U_{jk}, g_{jk})\}$ holds. Moreover, 
\[
\max_j |f_j|\leq K_{\rm const}(F)\cdot \max_{j, k}|g_{jk}|
\]
holds. 
\end{lemma}

\begin{proof}
The lemma follows from the definition of $K(F)$ and $K_{\rm const}(F)$, 
since it follows from $H^0(Y, \mathcal{O}_Y(F))=0$ and $H^0(Y, \mathbb{C}(F))=0$ that the solution of $\{(U_j, f_j)\}$ of $\delta$-equation is unique. 
\end{proof}

Note that we can replace $\check{B}^1(\{U_j\}, \mathcal{O}_Y(F))$ and $\check{B}^1(\{U_j\}, \mathbb{C}(F))$ in Lemma \ref{lem:ueda_type_estim_K_Kconst_explicitly} with $\check{Z}^1(\{U_j\}, \mathcal{O}_Y(F))$ and $\check{Z}^1(\{U_j\}, \mathbb{C}(F))$, respectively, under our assumption that $Y$ is an elliptic curve by virtue of the following:
\begin{lemma}\label{lem:coh_of_ellipt_curve_OCF}
For any $F\in {\rm Pic}^0(Y)\setminus\{\mathbb{I}_Y\}$, $\check{H}^1(\{U_j\}, \mathcal{O}_Y(F))=\check{H}^1(\{U_j\}, \mathbb{C}(F))=0$. 
\end{lemma}

\begin{proof}
Let $F$ be an element of ${\rm Pic}^0(Y)\setminus\{\mathbb{I}_Y\}$. 
Then the vanishing of $\check{H}^1(\{U_j\}, \mathcal{O}_Y(F))\ (=H^1(Y, \mathcal{O}_Y(F)))$ follows from $H^0(Y, \mathcal{O}_Y(F))=0$ and Riemann–Roch theorem. Consider the short exact sequence 
\[
0 \to \mathbb{C}(F) \to \mathcal{O}_Y(F)\oplus \overline{\mathcal{O}_Y}(F) \to \mathcal{H}_Y(F)\to 0
\]
of sheaves on $Y$, 
where $\overline{\mathcal{O}_Y}(F)$ and $\mathcal{H}_Y(F)$ are the sheaves of antiholomoprhic and pluriharmonic sections of $F$, respectively. 
From this short exact sequence, it follows from a simple argument that 
\[
0\to \check{H}^1(\{U_j\}, \mathbb{C}(F)) \to \check{H}^1(\{U_j\}, \mathcal{O}_Y(F)) \oplus \check{H}^1(\{U_j\}, \overline{\mathcal{O}_Y}(F)) 
\]
is exact (See \cite[\S 1.2]{U} for the details). As $\check{H}^1(\{U_j\}, \mathcal{O}_Y(F))=0$ and 
\[
\check{H}^1(\{U_j\}, \overline{\mathcal{O}_Y}(F)) =\overline{\check{H}^1(\{U_j\}, \mathcal{O}_Y(\overline{F}))}=0
\]
holds (here we again applied Riemann–Roch theorem to the non-trivial unitary flat line bundle $\overline{F}$), the assertion holds. 
\end{proof}

From Ueda's estimate \cite[Lemma 4]{U}, we have the following: 
\begin{proposition}\label{prop:ueda_lemma}
There exist positive constants $K_1$ and $K_2$ such that  
\[
\frac{K_1}{d(\mathbb{I}_Y, F)} \leq K_{\rm const}(F) \leq K(F) \leq \frac{K_2}{d(\mathbb{I}_Y, F)}, 
\]
holds for any $F\in {\rm Pic}^0(Y)\setminus\{\mathbb{I}_Y\}$. 
\end{proposition}

\begin{proof}
Let us denote by $d_{\rm Ueda}$ the invariant distance on ${\rm Pic}^0(Y)$ which is constructed in \cite[\S 4.5]{U}. Note that 
\begin{equation}\label{eq:inv_distance}
d_{\rm Ueda}(\mathbb{I}_Y, F) = \inf\left\{ \left.\max_{j, k}|f_j-s_{jk}\cdot f_k| \right| \{(U_j, f_j)\}\in \check{C}^0(\{U_j\}, {\rm U}(1))\right\}
\end{equation}
holds for any $F= [\{(U_{jk}, s_{jk})\}]\in \check{H}^1(\{U_j\}, {\rm U}(1))$. 
As $d_{\rm Ueda}$ is Lipschitz equivalent to our distance $d$ (see \cite[\S A.3]{KU}), it is sufficient to show the assertion by replacing $d$ with $d_{\rm Ueda}$. 

The inequality $K_{\rm const}(F) \leq K(F)$ follows directly from their definitions. 
The existence of $K_2$ such that $K(F) \leq K_2/d_{\rm Ueda}(\mathbb{I}_Y, F)$ follows from \cite[Lemma 4]{U}. 
Therefore we will show the existence of $K_1$ such that $K_1/d_{\rm Ueda}(\mathbb{I}_Y, F) \leq K_{\rm const}(F)$ holds in what follows. 

Take $F= [\{(U_{jk}, s_{jk})\}]\in \check{H}^1(\{U_j\}, {\rm U}(1))$. 
By the equation (\ref{eq:inv_distance}), there exists an element $\{(U_j, f_j)\}\in \check{C}^0(\{U_j\}, {\rm U}(1))$ such that 
\[
\max_{j, k}|f_j-s_{jk}\cdot f_k| < 2\cdot d_{\rm Ueda}(\mathbb{I}_Y, F)
\]
holds. 
Then, by regarding $\{(U_j, f_j)\}$ as an element of $\check{C}^0(\{U_j\}, \mathbb{C}(F))$, it follows from the definition of $K_{\rm const}(F)$ that
\[
1 = \max_j |f_j| \leq  K_{\rm const}(F)\cdot \max_{j, k}|f_j-s_{jk}\cdot f_k|<K_{\rm const}(F)\cdot 2\cdot d_{\rm Ueda}(\mathbb{I}_Y, F)
\]
holds, from which the assertion follows.  
\end{proof}

\section{Some consequences from $K_V^{-1}=[Y]$}\label{section:z_ext_K=-Y}

We use the notation in \S \ref{section:kigou_settei}. 
In this section, we investigate the groups $H^0(V, \mathcal{O}_V/\mathcal{C}_{\mathcal{F}})$ and $H^0(V^*, \mathcal{O}_{V^*}/\mathcal{C}_{\mathcal{F}}^*)$ mainly by using the assumption that $K_V^{-1}=[Y]$. In what follows, let us fix a meromorphic $2$-form $\eta$ whose divisor is $-Y$. 

First let us show the following proposition for $H^0(V, \mathcal{O}_V/\mathcal{C}_{\mathcal{F}})$. 
\begin{proposition}\label{prop:H0_of_O/CF}
By changing the manner how to extend $z_j$ form $U_j$ to $V_j$ if necessary, the following holds: \\
$(i)$ It holds that $-z_j+z_k\in \Gamma(V_{jk}, \mathcal{C}_{\mathcal{F}})$ for each $j$ and $k$. \\
$(ii)$ $H^0(V, \mathcal{O}_V/\mathcal{C}_{\mathcal{F}})$ is a vector space of dimension $1$ which is generated by $\{(V_j, z_j\ {\rm mod}\ \mathcal{C}_{\mathcal{F}})\}$, 
where $z_j\ {\rm mod}\ \mathcal{C}_{\mathcal{F}}$ denotes the image of $z_j$ by the natural map $\Gamma(V_j, \mathcal{O}_V)\to \Gamma(V_j, \mathcal{O}_V/\mathcal{C}_{\mathcal{F}})$. 
\end{proposition}

\begin{proof}
Let $h_j(z_j, w_j)$ be the holomorphic function on $V_j$ such that 
\[
\eta|_{V_j} = h_j\cdot \frac{dz_j\wedge dw_j}{w_j}
\]
holds. As $dw_j/w_j=dw_k/w_k$ holds on each $V_{jk}$ and $dz_j=dz_k$ holds on each $U_{jk}$, one has that $\{(U_j, h_j(z_j, 0))\}$ patches to define a global holomorphic function on $Y$. 
Therefore there exists a constant $c\in \mathbb{C}$ such that $h_j(z_j, 0)\equiv c$ holds on each $U_j$. As $\eta$ has a pole of order $1$ along $Y$, $c\not=0$. 
Thus we may assume that $c=1$ by replacing $\eta$ with $\eta/c$ if necessary. 

Fixing a point $z_j^0\in U_j$ for each $j$, let us consider the function $\widehat{z}_j$ on $V_j$ defined by 
\[
\widehat{z}_j(z_j, w_j) := z_j^0 + \int_{z_j^0}^{z_j}h_j(\zeta, w_j)\,d\zeta. 
\]
Then $\widehat{z}_j$ is a holomorphic function which satisfies $\widehat{z}_j|_{U_j}=z_j$. 
By construction, 
\[
d\widehat{z}_j \wedge \frac{dw_j}{w_j} = \eta|_{V_j}
\]
holds on each $V_j$. 
Thus we have that 
\[
d(\widehat{z}_j-\widehat{z}_k) \wedge \frac{dw_j}{w_j} = \eta_j-\eta_k \equiv 0
\]
holds on each $V_{jk}$ again by the equation $dw_j/w_j=dw_k/w_k$. 
Therefore it follows that
\[
\frac{\del}{\del z_j} (\widehat{z}_j-\widehat{z}_k) \equiv 0
\]
holds on each $V_{jk}$, from which the assertion $(i)$ holds by replacing $z_j$ with $\widehat{z}_j$. 

For proving the assertion $(ii)$, let us consider the morphism $\mathcal{O}_V\to \mathcal{O}_V(K_V\otimes [-Y])$ of sheaves defined by 
\[
\Gamma(W, \mathcal{O}_V)\ni f\mapsto df\wedge \zeta|_W \in\Gamma(W, \mathcal{O}_V(K_V\otimes [-Y]))
\]
on each open subset $W\subset V$, where $\zeta:=\{(V_j, dw_j/w_j)\}$. 
By the exactness of 
$0\to \mathcal{C}_{\mathcal{F}}\to \mathcal{O}_V\to \mathcal{O}_V(K_V\otimes [-Y])$, 
the morphism above induces a morphism $\mathcal{O}_V/\mathcal{C}_{\mathcal{F}}\to \mathcal{O}_V(K_V\otimes [-Y])$, from which we obtain a linear map 
\[
F\colon H^0(V, \mathcal{O}_V/\mathcal{C}_{\mathcal{F}})\to H^0(V, \mathcal{O}_V(K_V\otimes [-Y])). 
\]
As 
\[
F(\{(V_j, \widehat{z}_j\ {\rm mod}\ \mathcal{C}_{\mathcal{F}})\}) = \eta
\]
holds by construction, it follows that $\{(V_j, \widehat{z}_j\ {\rm mod}\ \mathcal{C}_{\mathcal{F}})\}$ is a non trivial element of $H^0(V, \mathcal{O}_V/\mathcal{C}_{\mathcal{F}})$ and that $F$ is not the zero map. 
Moreover, as $H^0(V, \mathcal{O}_V(K_V\otimes [-Y]))=\mathbb{C}\cdot \eta$ follows from Lemma \ref{lem:nononconstantholfunctionexistsonV}, $F$ is a surjective linear map onto a $1$-dimensional space. 
Thus the assertion holds, since it can be easily checked that $F$ is injective. 
\end{proof}

Similarly we have the following for $H^0(V^*, \mathcal{O}_{V^*}/\mathcal{C}_{\mathcal{F}}^*)$.
\begin{proposition}\label{prop:H0_of_O/CF_*version}
For $z_j$'s as in 
Proposition \ref{prop:H0_of_O/CF}, 
$H^0(V^*, \mathcal{O}_{V^*}/\mathcal{C}_{\mathcal{F}}^*)$ is a vector space of dimension $1$ which is generated by $\{(V_j^*, z_j\ {\rm mod}\ \mathcal{C}_{\mathcal{F}})\}$. 
\end{proposition}

\begin{proof}
The proposition is shown by the same argument as in the proof of 
Proposition \ref{prop:H0_of_O/CF} $(ii)$. 
\end{proof}

In \S 4 and 5, we always assume that $z_j$'s are as in Proposition \ref{prop:H0_of_O/CF}. 

\section{Cohomologies of the sheaves of holomorphic functions constant along $\mathcal{F}$}

We use the notation in \S \ref{section:kigou_settei}. 
In this section, we investigate the groups $\check{H}^1(\{V_j\}, \mathcal{C}_{\mathcal{F}})$ and $\check{H}^1(\{V_j^*\}, \mathcal{C}_{\mathcal{F}}^*)$, and the map between them. 
Here our arguments are mainly based on Proposition \ref{prop:ueda_lemma} and the fact that $\check{H}^1(\{U_j\}, \mathbb{C}(N_{Y/X}^n))=0$ holds for any non-zero integer $n$, since we are assuming that $N_{Y/X}$ is non-torsion (Recall Lemma \ref{lem:coh_of_ellipt_curve_OCF}). 

\subsection{Preliminary observation on formal primitives}

First let us show the following: 
\begin{lemma}\label{lem:formal_solution_of_g}
Take an element 
$\{(V_{jk}, g_{jk})\}\in \check{Z}^1(\{V_j\}, \mathcal{C}_{\mathcal{F}})$. 
Let 
\[
g_{jk}(w_j) = \sum_{n=0}^\infty g_{jk, n}\cdot w_j^n
\]
be the Taylor expansion of $g_{jk}$. Assume that $g_{jk, 0}=0$ for each $j$ and $k$. 
Then there uniquely exists $f_{j, n}\in\mathbb{C}$ for each $j$ and $n\in\mathbb{Z}_{>0}$ such that the formal power series 
\[
f_j(w_j) := \sum_{n=1}^\infty f_{j, n}\cdot w_j^n
\]
satisfies $-f_j+f_k=g_{jk}$ formally. Such $f_{j, n}$'s satisfy 
\[
\max_j|f_{j, n}|\leq K_{\rm const}(N_{Y/X}^{-n})\cdot \max_{j, k}|g_{jk, n}|. 
\]
Moreover, the following conditions are equivalent to each other: \\ 
$(i)$ For any $j$, $f_j$ is a convergent power series whose radius of convergence is larger than or equal to $1$. \\
$(ii)$ $[\{(V_{jk}, g_{jk})\}]=0\in \check{H}^1(\{V_j\}, \mathcal{C}_{\mathcal{F}})$. 
\end{lemma}

\begin{proof}
As is simply observed by calculating the Taylor expansion of $g_{jk}+g_{k\ell}+g_{\ell j}$ on each $V_j\cap V_k\cap V_\ell$, it holds that $\{(U_{jk}, g_{jk, n})\}\in \check{Z}^1(\{U_j\}, \mathbb{C}(N_{Y/X}^{-n}))$ for each $n\in \mathbb{Z}_{>0}$. 
Therefore it follows from Lemmata \ref{lem:ueda_type_estim_K_Kconst_explicitly} and  \ref{lem:coh_of_ellipt_curve_OCF} that there uniquely exists $f_{j, n}$'s with the estimate as in the assertion such that $\delta(\{(U_j, f_{j, n})\})=\{(U_{jk}, g_{jk, n})\}\in \check{Z}^1(\{U_j\}, \mathbb{C}(N_{Y/X}^{-n}))$ holds. 
The equation $-f_j+f_k=g_{jk}$ can be easily checked. 
For proving the uniqueness, take a formal power series 
\[
h_{j}(w_j) = \sum_{n=0}^\infty h_{j, n}\cdot w_j^n
\]
such that $-h_j+h_k=g_{jk}$ holds formally. 
Then, again by comparing the coefficients, it follows that $\delta(\{(U_j, h_{j, n})\})=\{(U_{jk}, g_{jk, n})\}\in \check{Z}^1(\{U_j\}, \mathbb{C}(N_{Y/X}^{-n}))$ holds. 
Therefore $f_{j, n}=h_{j, n}$ holds for each $j$ and $n$ by Lemma \ref{lem:ueda_type_estim_K_Kconst_explicitly}. 
See also Remark \ref{rmk:keisikikai_C_F_or_holtubnbhd} below for these arguments. 

For the latter half of the assertion, it is clear that $(i)$ implies $(ii)$. 
Assume that $(ii)$ holds. 
Then there exists a $0$-cochain $\{(V_j, h_j)\}\in\check{C}^0(\{V_j\}, \mathcal{C}_{\mathcal{F}})$ such that $\delta(\{(V_j, h_j)\})=\{(V_{jk}, g_{jk})\}$. 
Let 
\[
h_{j}(w_j) = \sum_{n=0}^\infty h_{j, n}\cdot w_j^n
\]
be the Taylor expansion of $h_{j}$. Note that each $h_j$ is a convergent power series whose radius of convergence is larger than or equal to $1$, since it is a section of $\mathcal{C}_{\mathcal{F}}$ on $V_j$. 
By the uniqueness of the formal solution, 
it follows that $f_{j, n}=h_{j, n}$ holds for each $j$ and $n$, from which the assertion follows. 
\end{proof}

\begin{remark}\label{rmk:keisikikai_C_F_or_holtubnbhd}
For the argument in the proof of Lemma \ref{lem:formal_solution_of_g}, 
note that the same type arguments makes sense also for the $1$-cocycles of the sections of $\mathcal{O}_V$ if $V$ is a holomorphic tubular neighborhood (See the proof of \cite[Lemma 3.1]{KU2} for example). 
For running this type of arguments, 
it is essential that the partial derivatives of $g_{jk}$'s enjoy the equation such as $\frac{\del^n}{\del w_k^n}g_{jk} = t_{jk}^n\cdot \frac{\del^n}{\del w_j^n}g_{jk}$. 
Such an equation never holds in general since
\[
\frac{\del}{\del w_k} = \frac{\del w_j}{\del w_k}\cdot \frac{\del}{\del w_j} + \frac{\del z_j}{\del w_k}\cdot \frac{\del}{\del z_j}
\]
and $\del z_j/\del w_k$ may be a non-trivial function, which is the reason why the assumption that $V$ is a holomorphic tubular neighborhood is needed (Note that, as is explained  in \S \ref{section:kigou_settei}, one can use $z_j$'s such that $z_j=A_{jk}+z_k$ holds on each $V_{jk}$ under the assumption that $V$ admits a holomorphic tubular neighborhood, which assures $\del z_j/\del w_k\equiv 0$). In our situation in Lemma \ref{lem:formal_solution_of_g}, $V$ does {\it not} admits a holomorphic tubular neighborhood in general. Instead, we are assuming that $\{(V_{jk}, g_{jk})\}\in \check{Z}^1(\{V_j\}, \mathcal{C}_{\mathcal{F}})$, by which we can recover the equation $\frac{\del^n}{\del w_k^n}g_{jk} = t_{jk}^n\cdot \frac{\del^n}{\del w_j^n}g_{jk}$. 
\end{remark}

Similarly we have the following: 
\begin{lemma}\label{lem:formal_solution_of_g_*version}
Take an element $\{(V_{jk}^*, g_{jk})\}\in \check{Z}^1(\{V_j^*\}, \mathcal{C}_{\mathcal{F}}^*)$. Let
\[
g_{jk}(w_j) = \sum_{n=-\infty}^\infty g_{jk, n}\cdot w_j^n
\]
bet the Laurent expansion of $g_{jk}$. 
Then there uniquely exists $f_{j, -n}\in\mathbb{C}$ for each $j$ and $n\in\mathbb{Z}_{>0}$ such that the formal power series 
\[
f_j(w_j) := \sum_{n=-\infty}^{-1} f_{j, n}\cdot w_j^{n}
\]
satisfies that the principal part of the series $-f_j+f_k-g_{jk}$ is zero. Such $f_{j, -n}$'s satisfy 
\[
\max_j|f_{j, -n}|\leq K_{\rm const}(N_{Y/X}^{n})\cdot \max_{j, k}|g_{jk, -n}|. 
\]
Moreover, the following conditions are equivalent to each other: \\
$(i)$ For any $j$, the power series $\sum_{n=1}^{\infty} f_{j, -n}\cdot X^n$ is convergent and its radius of convergence is $\infty$. \\
$(ii)$ $[\{(V_{jk}^*, g_{jk})\}]\in {\rm Image}(\check{H}^1(\{V_j\}, \mathcal{C}_{\mathcal{F}})\to \check{H}^1(\{V_j^*\}, \mathcal{C}_{\mathcal{F}}^*))$. 
\end{lemma}

\begin{proof}
The assertion can be shown by the same argument as in the proof of Lemma \ref{lem:formal_solution_of_g}. 
\end{proof}

In the rest of this subsection, we show the following proposition, which is important for investigating Case I and I\!I from the viewpoint of Remark \ref{rmk:keisikikai_C_F_or_holtubnbhd}
\begin{proposition}\label{prop:hol_tub_nbhd_aru}
Assume that the map $\varinjlim_{W}\check{H}^1(\{W\cap V_j\}, \mathcal{C}_{\mathcal{F}})\to \check{H}^1(\{U_j\}, \mathbb{C})$ induced from the restriction maps is a linear bijection, where the direct limit is taken over the set of all the open neighborhoods $W$ of $Y$ in $V$. 
Then there exist an open neighborhood $W$ of $Y$ in $V$ 
and a holomorphic function $\widehat{z}_j\colon W\cap V_j\to \mathbb{C}$ with $\widehat{z}_j|_{U_j}=z_j$ for each $j$ such that $\widehat{z}_j-\widehat{z}_k\equiv A_{jk}$ holds on each $V_{jk}$. 
Especially, in Case I and I\!I, we may assume that $V$ is a holomorphic tubular neighborhood of $Y$ and 
$z_j-z_k\equiv A_{jk}$ on each $V_{jk}$ by shrinking $V$, changing the scaling of $w_j$'s, and by replacing $z_j$ with the function $\widehat{z}_j$ as above. 
\end{proposition}

\begin{proof}
Set $g_{jk}:=-z_j+z_k$ on each $V_{jk}$. 
Then it follows from Proposition \ref{prop:H0_of_O/CF} $(i)$ that we may assume 
$\{(V_{jk}, g_{jk})\}\in \check{H}^1(\{V_j\}, \mathcal{C}_{\mathcal{F}})$. 
As $g_{jk}|_{U_{jk}}\equiv A_{jk}$, it follows from the assumption that 
there exist an open neighborhood $W$ of $Y$ in $V$ and $\{(W\cap V_j, f_j)\}\in\check{C}^0(\{W\cap V_j\}, \mathcal{C}_{\mathcal{F}})$ such that 
$-f_j+f_k=-A_{jk}+g_{jk}$ holds on each $W\cap V_{jk}$. 
Note that $(-f_j+f_k)|_{U_{jk}}\equiv 0$ by construction, which means that $\{(U_j, f_j)\}$ patches to define a holomorphic function on $Y$. Therefore there exists a constant $c\in \mathbb{C}$ such that $f_j|_{U_j}\equiv c$ holds for any $j$. 
Thus, the desired function can be obtained by letting $\widehat{z}_j:=z_j-f_j+c$. 
The latter half of the assertion follows from Lemma \ref{lem:directlim_IorII_forC_F} below. 
\end{proof}

\begin{lemma}\label{lem:directlim_IorII_forC_F}
In Case I and I\!I, the map $\varinjlim_{W}\check{H}^1(\{W\cap V_j\}, \mathcal{C}_{\mathcal{F}})\to \check{H}^1(\{U_j\}, \mathbb{C})$ as in Proposition \ref{prop:hol_tub_nbhd_aru} is a linear bijection. 
\end{lemma}

\begin{proof}
As the surjectivity is clear, we show the injectivity of this map. 
Take an open neighborhood $W$ of $Y$ in $V$ and an element $\{(W\cap V_{jk}, g_{jk})\}\in\check{Z}^1(\{W\cap V_{j}\}, \mathcal{C}_{\mathcal{F}})$ such that $[\{(U_{jk}, g_{jk}|_{U_{jk}})\}]=0\in \check{H}^1(\{U_j\}, \mathbb{C})$. 
Let 
\begin{equation}\label{eq:firstone_intheproofof_lem:directlim_IorII_forC_F}
g_{jk}(w_j) = \sum_{n=0}^\infty g_{jk, n}\cdot w_j^n
\end{equation}
be the Taylor expansion. As $[\{(U_{jk}, g_{jk, 0})\}]=0\in \check{H}^1(\{U_j\}, \mathbb{C})$, one can take $\{(U_j, f_{j, 0})\}\in\check{C}^0(\{U_j\}, \mathbb{C})$ such that $\delta(\{(U_j, f_{j, 0})\})=\{(U_{jk}, g_{jk, 0})\}\in \check{Z}^1(\{U_j\}, \mathbb{C})$. 
Then it follows from Lemmata \ref{lem:ueda_type_estim_K_Kconst_explicitly} and  \ref{lem:coh_of_ellipt_curve_OCF}, and by the same argument as in the proof of Lemma \ref{lem:formal_solution_of_g}, that 
there uniquely exists $f_{j, n}\in\mathbb{C}$ for each $j$ and $n\in\mathbb{Z}_{>0}$ such that 
\begin{equation}\label{eq:secondone_intheproofof_lem:directlim_IorII_forC_F}
\max_j|f_{j, n}|\leq K_{\rm const}(N_{Y/X}^{-n})\cdot \max_{j, k}|g_{jk, n}|
\end{equation}
and that the formal power series 
\[
f_j(w_j) := \sum_{n=0}^\infty f_{j, n}\cdot w_j^n
\]
satisfies $-f_j+f_k=g_{jk}$ formally. 

Take an interior point $z_{jk}^0\in U_{jk}$ for each $j$ and $k$. Set 
\[
\ve := \frac{1}{2}\cdot \min_{j, k}\sup\{r>0\mid \{(z_{jk}^0, w_j)\mid |w_j|\leq r\}\subset W\cap V_{jk}\}, 
\]
which is clearly positive. 
Note that the radius of convergence of the power series (\ref{eq:firstone_intheproofof_lem:directlim_IorII_forC_F}) is larger than or equal to $2\ve$ for each $j$ and $k$. 
Therefore it follows from Cauchy's inequality and  that there exists a positive constant $C$ such that 
\[
\max_{j, k}|g_{jk, n}|\leq \frac{C}{\ve^n}
\]
holds for each $j$, $k$, and $n\in\mathbb{Z}_{>0}$. 
Thus, by definition of Case I and I\!I, the estimate (\ref{eq:secondone_intheproofof_lem:directlim_IorII_forC_F}), and Proposition \ref{prop:ueda_lemma}, we have the existence of positive numbers $C'$ and $R>1$ such that 
\[
|f_{j, n}|\leq \frac{C'}{(\ve/R)^n}
\]
holds for each $j$ and $n\in\mathbb{Z}_{>0}$. 
As a consequence, we have that $\{(W'\cap V_j, f_j)\}$ is an element of $\check{C}^0(\{W'\cap V_j\}, \mathcal{C}_{\mathcal{F}})$ which satisfies $\delta(\{(W'\cap V_j, f_j)\})=\{(W'\cap V_{jk}, g_{jk}|_{W'\cap V_{jk}})\}$, where $W':=\textstyle\bigcup_j\{|w_j|<\ve/R\}$. 
Therefore it follows that the element of $\varinjlim_{W}\check{H}^1(\{W\cap V_j\}, \mathcal{C}_{\mathcal{F}})$ defined by $g_{jk}$'s is trivial, from which the assertion follows. 
\end{proof}

Note that, by using the chart $\{(V_j, (z_j, w_j))\}$ as in Remark \ref{rmk:niceVjs_whenCaseI_II}, one can also show the following variant of Lemma \ref{lem:directlim_IorII_forC_F}. 
\begin{lemma}\label{lem:case2_kinou_lim_zero}
In Case I and I\!I, the natural map $\varinjlim_{W}H^1(W, \mathcal{O}_W)\to H^1(Y, \mathcal{O}_Y)$ is a linear bijection. 
Especially, in Case I and I\!I, the following holds for any holomorphic line bundle $L$ on $V$: 
There exists a holomorphic tubular neighborhood $W$ of $Y$ such that $L|_W\cong p^*(L|_Y)$ holds, where $p\colon W\to Y$ is the holomorphic retraction which corresponds to the projection $N_{Y/V}\to Y$. 
\end{lemma}

\begin{proof}
By using the chart $\{(V_j, (z_j, w_j))\}$ as in Remark \ref{rmk:niceVjs_whenCaseI_II}, 
it follows from the same argument as in the proof of Lemma \ref{lem:directlim_IorII_forC_F} (and the proof of \cite[Lemma 3.1]{KU2}) that the natural map $\varinjlim_{W}\check{H}^1(\{W\cap V_j\}, \mathcal{O}_W)\to \check{H}^1(\{U_j\}, \mathcal{O}_Y)$ is a linear bijection (See also Remark \ref{rmk:keisikikai_C_F_or_holtubnbhd}). Therefore the former half of the assertion holds. 
In order to show the latter half of the assertion, take a holomorphic line bundle $L$ on $V$. 
Set $F:=L\otimes p^*(L|_Y)$, where $p\colon V\to Y$ is the retraction as in the assertion. 
Note that $F|_Y$ is holomorphically trivial by definition, and that $F$ is topologically trivial on $V$, since we may assume that $V$ is a deformation retract of $Y$ by shrinking $V$ if necessary. 
Thus, by comparing the long exact sequences which come from the exponential short exact sequences $0\to \mathbb{Z}\to \mathcal{O}_Y\to \mathcal{O}_Y^*\to 0$ on $Y$ and $0\to \mathbb{Z}\to \mathcal{O}_V\to \mathcal{O}_V^*\to 0$ on $V$, one has that there exists an element $\xi \in H^1(V, \mathcal{O}_V)$ such that $e(\xi)=F$ and $\xi|_Y=0$, where $e\colon H^1(V, \mathcal{O}_V)\to H^1(V, \mathcal{O}_V^*)$ is the map induced from the exponential map. 
By the former half of the assertion, $\xi|_W=0\in H^1(W, \mathcal{O}_W)$ holds for a neighborhood $W$ of $Y$ in $V$. Thus it follows that $F|_W$ is holomorphically trivial, from which the assertion follows. 
\end{proof}

\subsection{Observation on $\check{H}^1(\{V_j\}, \mathcal{C}_{\mathcal{F}})$}

In this subsection, we show the following: 
\begin{proposition}\label{prop:main_for_H^1_of_C_F}
The following holds: \\
$(i)$ In Case I, by using the chart $\{(V_j, (z_j, w_j))\}$ as in Remark \ref{rmk:niceVjs_whenCaseI_II}, the following holds: For any $\mathfrak{g}=\{(V_{jk}, g_{jk})\}\in \check{Z}^1(\{V_j\}, \mathcal{C}_{\mathcal{F}})$ with $g_{jk}|_{U_{jk}}\equiv 0$, $\mathfrak{g}\in \check{B}^1(\{V_j\}, \mathcal{C}_{\mathcal{F}})$ holds. 
Especially, the map $\check{H}^1(\{V_j\}, \mathbb{C})\to \check{H}^1(\{V_j\}, \mathcal{C}_{\mathcal{F}})$ induced from the natural morphism $\mathbb{C}\to \mathcal{C}_{\mathcal{F}}$ is a linear bijection. \\
$(ii)$ In Case I\!I and I\!I\!I, $\check{H}^1(\{V_j\}, \mathcal{C}_{\mathcal{F}})$ is of non-Hausdorff type: i.e. it is of infinite dimension and not Hausdorff. 
\end{proposition}

Here the topology of $\check{H}^1(\{V_j\}, \mathcal{C}_{\mathcal{F}})$ is defined in the same manner as 
in Lemma \ref{lem:comparizon_top_nonhausdoeff_CDcorresp}: 
i.e. by regarding the space $\check{C}^{1}(\{V_j\}, \mathcal{C}_{\mathcal{F}})$ as a Fr\'echet space by using the topology of uniform convergence on compact sets, by topologizing the space $\check{B}^1(\{V_j\}, \mathcal{C}_{\mathcal{F}})$ and $\check{Z}^1(\{V_j\}, \mathcal{C}_{\mathcal{F}})$ by the relative topologies as subspaces of $\check{C}^{1}(\{V_j\}, \mathcal{C}_{\mathcal{F}})$, 
we are regarding $\check{H}^1(\{V_j\}, \mathcal{C}_{\mathcal{F}})$ as the topological vector space obtained by the quotient $\check{Z}^1(\{V_j\}, \mathcal{C}_{\mathcal{F}})/\check{B}^1(\{V_j\}, \mathcal{C}_{\mathcal{F}})$. 

\begin{proof}[Proof of Proposition \ref{prop:main_for_H^1_of_C_F}]
Recall that, by Proposition \ref{prop:hol_tub_nbhd_aru}, 
we can take the chart $\{(V_j, (z_j, w_j))\}$ as in Remark \ref{rmk:niceVjs_whenCaseI_II} in Case I. 
Take $\mathfrak{g}=\{(V_{jk}, g_{jk})\}\in \check{Z}^1(\{V_j\}, \mathcal{C}_{\mathcal{F}})$ with $g_{jk}|_{U_{jk}}\equiv 0$. Let 
\[
g_{jk}(w_j) = \sum_{n=1}^\infty g_{jk, n}\cdot w_j^n
\]
be the Taylor expansion. 
By Lemma \ref{lem:formal_solution_of_g}, it is sufficient to show that the formal power series 
\[
f_j(w_j) := \sum_{n=1}^\infty f_{j, n}\cdot w_j^n
\]
as in the lemma has the radius of convergence larger than or equal to $1$, 
which follows from the estimate $|f_{j, n}|\leq K_{\rm const}(N_{Y/X}^{-n})\cdot \max_{j, k}|g_{jk, n}|$ and Lemma \ref{prop:ueda_lemma} (Here we again use the assumption that we are in Case I). 
Thus the assertion $(i)$ holds, since the the latter half is a simple consequence from the former half of the assertion. The assertion $(ii)$ follows from Lemma \ref{lemma:for_non_hausdorff_case} below. 
\end{proof}

\begin{lemma}\label{lemma:for_non_hausdorff_case}
In Case I\!I and I\!I\!I, there exist countably many linearly independent elements 
$\{[\mathfrak{g}^{(\lambda)}]\}_{\lambda=1}^\infty$ of the closure $K:=\overline{\{0\}}$ of $\{0\}$ in $\check{H}^1(\{V_j\}, \mathcal{C}_{\mathcal{F}})$. 
\end{lemma}

\begin{proof}
In Case I\!I and I\!I\!I, it follows from Proposition \ref{prop:ueda_lemma} that there exist a positive number 
$R>1$ and a strictly increasing sequence $\{n_\nu\}_{\nu=1}^\infty$ of positive integers such that 
$K_{\rm const}(N_{Y/X}^{-n_\nu}) > R^{n_\nu}$ holds. 
By definition of $K_{\rm const}(N_{Y/X}^{-n_\nu})$, for each $j$ and $n\in\mathbb{Z}_{>0}$, there exists 
$f_{j, n}\in\mathbb{C}$ such that 
\[
\max_j|f_{j, n}|> \frac{1}{2} K_{\rm const}(N_{Y/X}^{-n})\cdot \max_{j, k}|g_{jk, n}|
\]
holds for $g_{jk, n}:=-f_{j, n}+t_{jk}^{-n}\cdot f_{k, n}$. 
Note that $\max_{j, k}|g_{jk, n}|>0$, since otherwise $f_{j, n}\equiv 0$ holds by the uniqueness of the $\delta$-equation (Lemma \ref{lem:ueda_type_estim_K_Kconst_explicitly}). 
Therefore, by multiplying a constant, we may assume that $\max_{j, k}|g_{jk, n}|\equiv 1$. Then it follows that 
\begin{equation}\label{estim_for_fj0_infinitely_many_nu_beforehatonosu}
\max_j|f_{j, n_\nu}|> \frac{1}{2}R^{n_\nu}
\end{equation}
holds for each $\nu$. 

For each positive integer $\lambda$, set 
\[
g_{jk}^{(\lambda)}:=\sum_{\mu\in 2^\lambda\cdot \mathbb{Z}_{>0}} g_{jk, n_\mu}\cdot w_j^{n_\mu}. 
\]
Then $\mathfrak{g}^{(\lambda)}:=\{(V_{jk}, g_{jk}^{(\lambda)})\}\in \check{Z}^1(\{V_j\}, \mathcal{C}_{\mathcal{F}})$ holds, since $\max_{j, k}|g_{jk, n}|\equiv 1$. 
As 
\[
\lim_{Q\to \infty}\delta\left(\left\{\left(V_j,\ \sum_{\ell=1}^Qf_{j, n_{\ell\cdot 2^\lambda}}\cdot w_j^{n_{\ell\cdot 2^\lambda}}\right)\right\}\right) 
=\mathfrak{g}^{(\lambda)}
\]
holds in $\check{Z}^1(\{V_j\}, \mathcal{C}_{\mathcal{F}})$, it follows that $[\mathfrak{g}^{(\lambda)}]\in K$ holds for each $\lambda\in\mathbb{Z}_{>0}$. 

In what follows we will show that $\{[\mathfrak{g}^{(\lambda)}]\}_{\lambda=1}^\infty$ is linearly independent by contradiction. 
Assume that there is a non-trivial linear relation. 
Then there exist $m, \mu\in\mathbb{Z}_{>0}$ and $(c_m, c_{m+1}, \dots, c_{m+\mu})\in \mathbb{C}^\mu$ such that $c_m\not=0$ and 
\[
\sum_{\lambda=m}^{m+\mu} c_\lambda\cdot \mathfrak{g}^{(\lambda)}\in \check{B}^1(\{V_j\}, \mathcal{C}_{\mathcal{F}})
\]
holds. Let $\widehat{c}_\nu\in \mathbb{C}$ be the constants such that 
\[
g_{jk}(w_j) := \sum_{\nu=1}^\infty \widehat{c}_{\nu}\cdot g_{jk, n_\nu}\cdot w_j^{n_\nu}
\]
satisfies $g_{jk}=\sum_{\lambda=m}^{m+\mu} c_\lambda\cdot g_{jk}^{(\lambda)}$. 
Note that such $\widehat{c}_\nu$'s can be determined from $c_m, c_{m+1}, \dots, c_{m+\mu}$: 
For example, $\widehat{c}_\nu=0$ holds for any positive integer $\nu$ which is not divisible by $2^m$, and $\widehat{c}_\nu=c_m$ holds for any positive integer $\nu$ which is divisible by $2^m$ and not divisible by $2^{m+1}$. 
By construction, it is clear that the formal solution $\{(V_j, f_j)\}$ of the equation $\delta(\{(V_j, f_j)\})=\{(V_{jk}, g_{jk})\}$ as in Lemma \ref{lem:formal_solution_of_g} is the one defined by 
\[
f_j(w_j) = \sum_{\nu=1}^\infty \widehat{c}_\nu\cdot f_{j, n_\nu}\cdot w_j^{n_\nu}. 
\]
By (\ref{estim_for_fj0_infinitely_many_nu_beforehatonosu}), the sets
\[
A_j := \left\{\nu\in (2^m\cdot \mathbb{Z}_{>0})\setminus (2^{m+1}\cdot \mathbb{Z}_{>0})\,\left|\,|f_{j, n_\nu}|> \frac{1}{2}R^{n_\nu}\right.\right\}
\]
satisfy 
\[
\bigcup_{j}A_j=(2^m\cdot \mathbb{Z}_{>0})\setminus (2^{m+1}\cdot \mathbb{Z}_{>0}), 
\]
from which it holds that $\# A_{j_0}=\infty$ holds for some $j_0$. 
As 
\[
\left|\widehat{c}_\nu\cdot f_{j_0, n_\nu}\right| = |c_m|\cdot |f_{j_0, n_\nu}|> |c_m|\cdot \frac{1}{2}R^{n_\nu}
\]
holds for any $\nu\in A_{j_0}$, 
one obtains that 
\[
R=\lim_{\nu\to \infty}\left(\frac{|c_m|}{2}R^{n_\nu}\right)^{\frac{1}{n_\nu}}
\leq \limsup_{\nu\to \infty}\left|\widehat{c}_\nu\cdot f_{j_0, n_\nu}\right|^{\frac{1}{n_\nu}}
\]
holds. 
Therefore it follows from Cauchy--Hadamard theorem that the radius of convergence of $f_{j_0}$ is less than or equal to $1/R$, which is less than $1$. Thus the contradiction follows from Lemma \ref{lem:formal_solution_of_g}. 
\end{proof}

\subsection{Observation on $\check{H}^1(\{V_j^*\}, \mathcal{C}_{\mathcal{F}}^*)$}

By a similar argument as in the previous section, we obtain the following: 
\begin{proposition}\label{prop:main_for_H^1_of_C_F_*version}
The following holds: \\
$(i)$ In Case I and I\!I, by using the chart $\{(V_j, (z_j, w_j))\}$ as in Remark \ref{rmk:niceVjs_whenCaseI_II}, the following holds: 
For any $\mathfrak{g}=\{(V_{jk}^*, g_{jk})\}\in \check{Z}^1(\{V_j^*\}, \mathcal{C}_{\mathcal{F}}^*)$, 
there exists an element $\mathfrak{f}=\{(V_j^*, f_j)\}\in \check{C}^0(\{V_j^*\}, \mathcal{C}_{\mathcal{F}}^*)$ such that $\mathfrak{g}-\delta(\mathfrak{f})\in {\rm Image}(\check{Z}^1(\{V_j\}, \mathcal{C}_{\mathcal{F}})\to \check{Z}^1(\{V_j^*\}, \mathcal{C}_{\mathcal{F}}^*))$ holds. 
Especially, the restriction map $\check{H}^1(\{V_j\}, \mathcal{C}_{\mathcal{F}})\to \check{H}^1(\{V^*_j\}, \mathcal{C}_{\mathcal{F}}^*)$ is a linear bijection. \\
$(ii)$ In Case I\!I\!I, the cokernel of the map $\check{H}^1(\{V_j\}, \mathcal{C}_{\mathcal{F}})\to \check{H}^1(\{V^*_j\}, \mathcal{C}_{\mathcal{F}}^*)$ is of non-Hausdorff type. 
\end{proposition}

Here the topology of $\check{H}^1(\{V_j^*\}, \mathcal{C}_{\mathcal{F}}^*)$ is the one defined by the same manner as that of $\check{H}^1(\{V_j\}, \mathcal{C}_{\mathcal{F}})$. 
We regard the cokernel of the map $\check{H}^1(\{V_j\}, \mathcal{C}_{\mathcal{F}})\to \check{H}^1(\{V^*_j\}, \mathcal{C}_{\mathcal{F}}^*)$ as a topological vector space by using the quotient topology. 

\begin{proof}[Proof of Proposition \ref{prop:main_for_H^1_of_C_F_*version}]
Recall again that, by Proposition \ref{prop:hol_tub_nbhd_aru}, 
we can take the chart $\{(V_j, (z_j, w_j))\}$ as in Remark \ref{rmk:niceVjs_whenCaseI_II} in Case I and I\!I. 
Take $\{(V_{jk}^*, g_{jk})\}\in \check{Z}^1(\{V_j^*\}, \mathcal{C}_{\mathcal{F}}^*)$. Let 
\[
g_{jk}(w_j) = \sum_{n=-\infty}^\infty g_{jk, n}\cdot w_j^n
\]
be the Laurent expansion. 
By Lemma \ref{lem:formal_solution_of_g_*version}, it is sufficient to show that, for the formal power series 
\[
f_j(w_j) = \sum_{n=-\infty}^{-1} f_{j, n}\cdot w_j^n
\]
as in the lemma, the radius of convergence of the power series $\sum_{n=1}^{\infty} f_{j, -n}\cdot X^n$ is infinity, 
which follows from the estimate $|f_{j, -n}|\leq K_{\rm const}(N_{Y/X}^{n})\cdot \max_{j, k}|g_{jk, -n}|$ and Lemma \ref{prop:ueda_lemma} (Here we again use the assumption that we are in Case I or I\!I). 
Thus the assertion $(i)$ holds, since the the latter half is a simple consequence from the former half of the assertion and Lemma \ref{lem:injectiveti_besweenH^1s_CF} below. The assertion $(ii)$ follows from Lemma \ref{lem:constr_of_g_lambdas} below. 
\end{proof}

\begin{lemma}\label{lem:injectiveti_besweenH^1s_CF}
The natural map $\check{H}^1(\{V_j\}, \mathcal{C}_{\mathcal{F}})\to \check{H}^1(\{V^*_j\}, \mathcal{C}_{\mathcal{F}}^*)$ is injective. 
\end{lemma}

\begin{proof}
Take an element $\{(V_{jk}, g_{jk})\}\in\check{Z}^1(\{V_j\}, \mathcal{C}_{\mathcal{F}})$ such that 
$[\{(V_{jk}, g_{jk})\}]=0$ holds as elements of $\check{H}^1(\{V^*_j\}, \mathcal{C}_{\mathcal{F}}^*)$. 
Then there exists $\{(V_j, f_j)\}\in \check{C}^0(\{V^*_j\}, \mathcal{C}_{\mathcal{F}}^*)$ such that 
$\delta(\{(V_j, f_j)\})=\{(V_{jk}, g_{jk})\}$ holds. Let 
\[
f_j(w_j) = \sum_{n=-\infty}^\infty f_{j, n}\cdot w_j^n
\]
be the Laurent expansion. 
Then $\{(U_j, f_{j, n})\}$ patches to define an element of $H^0(Y, \mathbb{C}(N_{Y/V}^{-n}))$ for each $n$ negative, which is equal to $0$ since $N_{Y/V}$ is non-torsion. Therefore one has that $f_{j, n}=0$ holds for each $j$ and any negative integer $n$, from which the assertion follows. 
\end{proof}

\begin{lemma}\label{lem:constr_of_g_lambdas}
In Case I\!I\!I, 
there exist countably many elements 
$\mathfrak{g}^{(\lambda)}=\{(V_{jk}^*, g_{jk}^{(\lambda)})\}\in \check{Z}^1(\{V_j^*\}, \mathcal{C}_{\mathcal{F}}^*)$ such that 
$\{[\mathfrak{g}^{(\lambda)}]\}_{\lambda=1}^\infty$ are linearly independent elements of the closure $\overline{\{0\}}$ of $\{0\}$ in $\check{H}^1(\{V_j^*\}, \mathcal{C}_{\mathcal{F}}^*)$. 
Moreover, they satisfy the following condition: 
Let 
\[
f_j^{(\lambda)}(w_j) = \sum_{n=-\infty}^{-1} f_{j, n}^{(\lambda)}\cdot w_j^{n}
\]
be the formal power series as in Lemma \ref{lem:formal_solution_of_g_*version} such that the principal part of $-f_j^{(\lambda)}+f_k^{(\lambda)}-g_{jk}^{(\lambda)}$ is zero. 
Then it holds that $-f_j^{(\lambda)}+f_k^{(\lambda)}=g_{jk}^{(\lambda)}$ formally, 
$|f_{j, -n}^{(\lambda)}|\leq 2^{-n}$ holds for each $j, \lambda$, and $n>0$, 
and that, for each $\lambda$, there exists $j_\lambda$ such that 
\[
\# \{n\in \mathbb{Z}_{>0}\,\mid\,|f_{j_\lambda, -n}^{(\lambda)}|=2^{-n}\}=\infty
\]
holds. Especially, $f_{j}^{(\lambda)}$ is convergent on $\{1/2<|w_j|<1\}$ for each $j$, whereas $f_{j_\lambda}^{(\lambda)}$ is not convergent on $V_{j_\lambda}^*$. 
\end{lemma}

\begin{proof}
Fix an increasing sequence $\{R_n\}_{n=1}^\infty$ of positive real numbers so that $\lim_{n\to \infty}R_n=\infty$ holds. 
In Case I\!I\!I, it follows from Proposition \ref{prop:ueda_lemma} that 
one can take a strictly increasing sequence $\{n_\nu\}_{\nu=1}^\infty$ of positive integers such that \[
K_{\rm const}(N_{Y/X}^{n_\nu}) > R_\nu^{n_\nu}
\]
holds for each $\nu$. 
By definition of $K_{\rm const}(N_{Y/X}^{n})$, for each $n$, there exist constants $f_{j, -n}$ such that 
\[
\max_j|f_{j, -n}|> \frac{1}{2} K_{\rm const}(N_{Y/X}^{n})\cdot \max_{j, k}|g_{jk, -n}|
\]
holds for $g_{jk, -n}:=-f_{j, -n}+t_{jk}^n\cdot f_{k, -n}$. 
By multiplying a constant, we may assume that 
\[
\max_j|f_{j, -n}|=\frac{1}{2^n}
\]
holds for each $n$. Note that 
\[
\max_{j, k}|g_{jk, -n_\nu}| < \frac{2}{K_{\rm const}(N_{Y/X}^{n_\nu})}\cdot \frac{1}{2^{n_\nu}} < 2\cdot \left(\frac{1}{2R_\nu}\right)^{n_\nu}
\]
holds for each $\nu$. For each positive integer $\lambda$, set 
\[
g_{jk}^{(\lambda)}:=\sum_{\mu\in 2^\lambda\cdot \mathbb{Z}_{>0}}^\infty g_{jk, -n_\mu}\cdot \frac{1}{w_j^{n_\mu}}. 
\]
Then the assertion can be shown 
in the same manner as in the proof of Lemma \ref{lemma:for_non_hausdorff_case}. 
\end{proof}

\section{$\delbar$ cohomologies of $V$ and $V^*$}

We use the notation in \S \ref{section:kigou_settei}. 
In this section, we show the following proposition by applying the results in \S 3 and 4. 
\begin{proposition}\label{prop:section5_main}
For a suitable choice of $V$, the following holds: \\
$(i)$ In Case I, both $H^1(V, \mathcal{O}_V)$ and $H^1(V^*, \mathcal{O}_{V^*})$ are vector space of dimension $1$ and the natural map 
$H^1(V, \mathcal{O}_V)\to H^1(V^*, \mathcal{O}_{V^*})$ is a linear bijection. \\
$(ii)$ In Case I\!I, both $H^1(V, \mathcal{O}_V)$ and $H^1(Y, \mathcal{O}_Y)$ are of non-Hausdorff type, and the natural map $H^1(V, \mathcal{O}_{V})\to H^1(V^*, \mathcal{O}_{V^*})$ is a linear bijection. \\
$(iii)$ Case I\!I\!I, both 
$H^1(V, \mathcal{O}_V)$ and $H^1(V^*, \mathcal{O}_{V^*})$ 
are of non-Hausdorff type. 
\end{proposition}

\begin{table}[h]
 \centering
  \begin{tabular}{c|ccc}
    & $H^1(V, \mathcal{O}_V)$ & $H^1(V^*, \mathcal{O}_{V^*})$ & $H^1(V, \mathcal{O}_V)\to H^1(V^*, \mathcal{O}_{V^*})$ \\
   \hline 
   Case I & $\cong \mathbb{C}$ & $\cong \mathbb{C}$ & bijective \\
   Case I\!I & non-H type & non-H type & bijective \\
   Case I\!I\!I & non-H type & non-H type & ?
  \end{tabular}
\vskip1mm
 \caption{Summary of the assertions of Proposition \ref{prop:section5_main}, non-H type means that it is of non-Hausdorff type. }
 \label{table:H^1_for_each_of_the_cases}
\end{table}

\subsection{Case I}

First let us investigate in Case I. 
Here we use the chart $\{(V_j, (z_j, w_j))\}$ as in Remark \ref{rmk:niceVjs_whenCaseI_II} (Recall Proposition \ref{prop:hol_tub_nbhd_aru}). 

\begin{proof}[Proof of Proposition \ref{prop:section5_main} $(i)$]
Take an element $\{(V_{jk}, g_{jk})\}\in \check{Z}^1(\{V_{jk}\}, \mathcal{O}_V)$. Let 
\[
g_{jk}(z_j, w_j) = \sum_{n=0}^\infty g_{jk, n}(z_j)\cdot w_j^n
\]
be the Taylor expansion of each $g_{jk}$ in the variable $w_j$. 
Then, by using Lemmata \ref{lem:ueda_type_estim_K_Kconst_explicitly} and \ref{lem:coh_of_ellipt_curve_OCF}, 
one can run the same argument as in the proof of Lemma \ref{lem:formal_solution_of_g} to obtain the solution $\{(U_j, f_{j, n})\}\in \check{C}^0(\{U_j\}, \mathcal{O}_Y(N_{Y/X}^{-n})$ of the equation $\delta(\{(U_j, f_{j, n})\})=\{(U_{jk}, g_{jk, n})\}\in \check{Z}^1(\{U_j\}, \mathcal{O}_Y(N_{Y/X}^{-n}))$ which satisfies the estimate
\[
\max_j\sup_{U_j}|f_{j, n}|\leq K(N_{Y/X}^{-n})\cdot \max_{j, k}\sup_{U_{jk}}|g_{jk, n}|
\]
for each $n>0$ (Note that here we used the fact that $V$ is a holomorphic tubular neighborhood, see Remark \ref{rmk:keisikikai_C_F_or_holtubnbhd}). 
As it follows from Proposition \ref{prop:ueda_lemma} and the estimate of $|g_{jk}(z_j)|$'s by Cauchy inequality that
\[
f_j(z_j, w_j) := \sum_{n=1}^\infty f_{j, n}(z_j)\cdot w_j^n, 
\]
which satisfies $-f_j+f_k=g_{jk}$ at least formally by construction, is convergent on $V_j$. 
Therefore one has that $p^*\colon H^1(Y, \mathcal{O}_Y)\to H^1(V, \mathcal{O}_V)$ is surjective, where $p\colon V\to Y$ is the holomorphic retraction which corresponds to the projection $N_{Y/V}\to Y$. 
As the injectivity is clear since $i^*\circ p^*\colon H^1(Y, \mathcal{O}_Y)\to H^1(Y, \mathcal{O}_Y)$ is the identity, where $i\colon Y\to V$ is the inclusion, $p^*$ is a linear bijection. 
Thus it follows that $H^1(V, \mathcal{O}_V)\cong\mathbb{C}$. 

From the same argument, it follows that $p^*\colon H^1(Y, \mathcal{O}_Y)\to H^1(V^*, \mathcal{O}_{V^*})$ is also surjective. 
Especially one has that ${\rm dim}\,H^1(V^*, \mathcal{O}_{V^*})\leq 1$. 
By Propositions \ref{prop:main_for_H^1_of_C_F} and \ref{prop:main_for_H^1_of_C_F_*version}, 
the natural maps $\check{H}^1(\{V_j\}, \mathbb{C})\to \check{H}^1(\{V_j\}, \mathcal{C}_{\mathcal{F}})$ and $\check{H}^1(\{V_j\}, \mathcal{C}_{\mathcal{F}})\to \check{H}^1(\{V^*_j\}, \mathcal{C}_{\mathcal{F}}^*)$ are linear bijections. 
Thus it follows from Proposition \ref{prop:H0_of_O/CF_*version} and the exact sequence (\ref{eq:main_exact_seq_*version}) that ${\rm dim}\,H^1(V^*, \mathcal{O}_{V^*})\geq 1$. 
Therefore $p^*\colon H^1(Y, \mathcal{O}_Y)\to H^1(V^*, \mathcal{O}_{V^*})$ is a surjective linear map between $1$-dimensional vector spaces, from which the assertion follows. 
\end{proof}

\subsection{Case I\!I}

Next let us investigate in Case I\!I. 
As a preparation, first let us show the following: 
\begin{lemma}\label{lem:K_1infin_then_K2infin}
The following holds: \\
$(i)$ Denote by $K_1$ the closure of $\{0\}$ in $\check{H}^1(\{V_j\}, \mathcal{C}_{\mathcal{F}})$, 
and by $K_2$ the closure of $\{0\}$ in $\check{H}^1(\{V_j\}, \mathcal{O}_V)$. 
If $K_1$ is of infinite dimension, then so is $K_2$. \\
$(ii)$ Denote by $K_1^*$ the closure of $\{0\}$ in $\check{H}^1(\{V_j^*\}, \mathcal{C}_{\mathcal{F}}^*)$, 
and by $K_2^*$ the closure of $\{0\}$ in $\check{H}^1(\{V_j^*\}, \mathcal{O}_{V^*})$. 
If $K_1^*$ is of infinite dimension, then so is $K_2^*$. 
\end{lemma}

\begin{proof}
As the natural maps $i\colon \check{H}^1(\{V_j\}, \mathcal{C}_{\mathcal{F}})\to \check{H}^1(\{V_j\}, \mathcal{O}_V)$ and
$j\colon \check{H}^1(\{V_j^*\}, \mathcal{C}_{\mathcal{F}}^*)\to \check{H}^1(\{V_j^*\}, \mathcal{O}_{V^*})$ 
are continuous, one has that both $i^{-1}(K_2)$ and $j^{-1}(K_2^*)$ are closed. 
Therefore $K_1\subset i^{-1}(K_2)$ and $K_1^*\subset j^{-1}(K_2^*)$ holds. 
As both of the kernels of the linear maps $i|_{K_1}\colon K_1\to K_2$ and $j|_{K_1^*}\colon K_1^*\to K_2^*$ are of finite dimension by the exact sequences (\ref{eq:main_exact_seq}) and (\ref{eq:main_exact_seq_*version}), and Propositions \ref{prop:H0_of_O/CF} and \ref{prop:H0_of_O/CF_*version}, the assertions hold. 
\end{proof}

By using the chart $\{(V_j, (z_j, w_j))\}$ as in Remark \ref{rmk:niceVjs_whenCaseI_II}, the assertion $(ii)$ of Proposition \ref{prop:section5_main} can be shown as follows: 
\begin{proof}[Proof of Proposition \ref{prop:section5_main} $(ii)$]
First let us show that $H^1(V, \mathcal{O}_V)$ is of non-Hausdorff type. 
From Lemmata \ref{lemma:for_non_hausdorff_case} and \ref{lem:K_1infin_then_K2infin}, 
it follows that the closure of $\{0\}$ in $\check{H}^1(\{V_j\}, \mathcal{O}_V)$ is of infinite dimension. 
Therefore the assertion follows from Lemma \ref{lem:comparizon_top_nonhausdoeff_CDcorresp}. 

Next let us show that the natural map $H^1(V, \mathcal{O}_{V})\to H^1(V^*, \mathcal{O}_{V^*})$ is a linear bijective. 
As the injectivity can be easily checked by the same argument as in the proof of 
Lemma \ref{lem:injectiveti_besweenH^1s_CF}, here we will show that it is surjective. 
Take an element $\{(V_{jk}^*, g_{jk})\}\in \check{Z}^1(\{V_{jk}^*\}, \mathcal{O}_{V^*})$. Let 
\[
g_{jk}(z_j, w_j) = \sum_{n=-\infty}^\infty g_{jk, n}(z_j)\cdot w_j^n
\]
be the Laurent expansion of each $g_{jk}$ in the variable $w_j$. 
Then, by using Lemmata \ref{lem:ueda_type_estim_K_Kconst_explicitly} and \ref{lem:coh_of_ellipt_curve_OCF}, 
one can run the same argument as in the proof of Lemma \ref{lem:formal_solution_of_g_*version} to obtain $\{(U_j, f_{j, -n})\}\in \check{C}^0(\{U_j\}, \mathcal{O}_Y(N_{Y/X}^{n}))$ such that 
\[
\max_j|f_{j, -n}|\leq K(N_{Y/X}^{n})\cdot \max_{j, k}|g_{jk, -n}|
\]
and that the formal power series 
\[
f_j(z_j, w_j) := \sum_{n=-\infty}^{-1} f_{j, n}(z_j)\cdot w_j^{n}
\]
satisfies that the principal part of the series $-f_j+f_k-g_{jk}$ is zero (Note that here we used the fact that $V$ is a holomorphic tubular neighborhood, see Remark \ref{rmk:keisikikai_C_F_or_holtubnbhd}). 
As it follows from Proposition \ref{prop:ueda_lemma} and the estimate of $|g_{jk}(z_j)|$'s by Cauchy inequality that each $f_j$ is convergent on $V_j$. 
Therefore $H^1(V, \mathcal{O}_{V})\to H^1(V^*, \mathcal{O}_{V^*})$ is a linear bijective. 

Finally, as $H^1(V, \mathcal{O}_{V})\to H^1(V^*, \mathcal{O}_{V^*})$ is a continuous injection from a space of non-Hausdorff type, the assertion follows. 
\end{proof}

\subsection{Case I\!I\!I}

\begin{proof}[Proof of Proposition \ref{prop:section5_main} $(iii)$]
From Lemmata \ref{lemma:for_non_hausdorff_case}, \ref{lem:constr_of_g_lambdas}, and \ref{lem:K_1infin_then_K2infin}, 
it follows that both the closures of $\{0\}$ in $\check{H}^1(\{V_j\}, \mathcal{O}_V)$ and $\check{H}^1(\{V_j^*\}, \mathcal{O}_{V^*})$ are of infinite dimension. 
Therefore the assertion follows from Lemma \ref{lem:comparizon_top_nonhausdoeff_CDcorresp}. 
\end{proof}

%

\section{Proof of Theorem \ref{thm:main}}

Let $\{M, V_1, V_2, \dots, V_N\}$ and $\mathcal{W}:=\{V_{\nu, j}^*\}_{\nu, j}\cup \{W_\alpha\}$ be open coverings of $X$ and $M$, respectively, as in 
\S \ref{section:prelim_case_non_tor_1compl_etc}. 
By Lemmata \ref{lem:first_lem_in_prelim} $(ii)$ and 
\ref{lem:nononconstantholfunctionexistsonV}, one obtains an exact sequence 
\[
0\to H^1(X, \mathcal{O}_X)\to H^1(M, \mathcal{O}_M)\oplus \bigoplus_{\nu=1}^NH^1(V_\nu, \mathcal{O}_{V_\nu})
\to \bigoplus_{\nu=1}^NH^1(V_\nu^*, \mathcal{O}_{V_\nu}^*) \to 0
\]
from the Mayer--Vietoris sequence for the open covering $\{M, V_1, V_2, \dots, V_N\}$. 

First let us consider the case where
\begin{description}
\item[{\bf (Condition)$_\nu$}] There exist positive numbers $\ve$ and $\delta$ such that, for any positive integer $n$, $d(\mathbb{I}_{Y_\nu}, N_{Y_\nu/X}^n)>\ve\cdot \delta^n$ holds. \qed
\end{description}
holds for any $\nu\in\{1, 2, \dots, N\}$. 
In this case, it follows from Proposition \ref{prop:section5_main} $(i)$, $(ii)$ that the map
$H^1(V_\nu, \mathcal{O}_{V_\nu})\to H^1(V_\nu^*, \mathcal{O}_{V_\nu}^*)$ is a linear bijection for each $\nu$ from which the assertion $(i)$ follows. 

Next let us consider the case where {\bf (Condition)$_\nu$} does not hold for some $\nu$. 
Without loss of generality, we may assume that {\bf (Condition)$_1$} does not hold. 
Take $\mathfrak{g}^{(\lambda)}=\{(V_{jk}^*, g_{jk}^{(\lambda)})\}$ and $\{(V_j^*, f_j^{(\lambda)})\}$ as in Lemma \ref{lem:constr_of_g_lambdas} for $V=V_1$ ($\lambda\in\mathbb{Z}_{>0}$). 
Set 
\[
\mathfrak{h}^{(\lambda)}:=\delta\left(\{(V_{1, j}^*, f_j^{(\lambda)})\}\cup \bigcup_{\nu=2}^N\{(V_{\nu, j}^*, 0)\}\cup \{(W_\alpha, 0)\}\right)
\]
formally. 
Then, as the component of $\mathfrak{h}^{(\lambda)}$ coincides with $g_{jk}^{(\lambda)}$ on each $V_{1, j}^*\cap V_{1, k}^*$, 
with $f_{j}^{(\lambda)}$ on each $V_{1, j}^*\cap W_{\alpha}$, which is convergent since $V_{1, j}^*\cap W_{\alpha}\subset \{1/2<|w_{1, j}|<1\}$, 
and with $0$ on the other types of the intersections of the elements of $\mathcal{W}$. 
Therefore one has that $\mathfrak{h}^{(\lambda)}\in \check{Z}^1(\mathcal{W}, \mathcal{O}_M)$. 
Note that 
\[
\lim_{\mu\to \infty}\delta\left(\{(V_{1, j}^*, f_j^{(\lambda, \mu)})\}\cup\bigcup_{\nu=2}^N\{(V_{\nu, j}^*, 0)\}\cup \{(W_\alpha, 0)\}\right) = \mathfrak{h}^{(\lambda)}
\]
holds for 
\[
f_j^{(\lambda, \mu)}(w_{1, j}) = \sum_{n=-\mu}^{-1}f_{j, n}^{(\lambda)}\cdot w_{1, j}^n, 
\]
from which one also has that $\mathfrak{h}^{(\lambda)}\in \overline{\check{B}^1(\mathcal{W}, \mathcal{O}_M)}$. 
Therefore it follows that $[\mathfrak{h}^{(\lambda)}]\in K$ holds for each $\lambda$, 
where $K$ is the closure of $\{0\}$ in $\check{H}^1(\mathcal{W}, \mathcal{O}_M)$. 

As is clear by construction, it holds that $i([\mathfrak{h}^{(\lambda)}])=[\mathfrak{g}^{(\lambda)}]$ for each $\lambda$, where $i\colon \check{H}^1(\mathcal{W}, \mathcal{O}_M)\to \check{H}^1(\{V_{1, j}^*\}, \mathcal{O}_{V_1^*})$ is the restriction map. 
As $\{i([\mathfrak{h}^{(\lambda)}])\}_{\lambda=1}^\infty$ is linearly independent in \linebreak$\check{H}^1(\{V_{1, j}^*\}, \mathcal{C}_{\mathcal{F}}^*)$, 
it follows from the exact sequence (\ref{eq:main_exact_seq_*version}) and Proposition \ref{prop:H0_of_O/CF_*version} that $\{i([\mathfrak{h}^{(\lambda)}])\}_{\lambda=1}^\infty$ spans an infinite dimensional subspace of $\check{H}^1(\{V_{1, j}^*\}, \mathcal{O}_{V_1^*})$. 
Therefore $\{[\mathfrak{h}^{(\lambda)}]\}_{\lambda=1}^\infty$ also spans an infinite dimensional subspace of $K$. 
Thus one has that $\check{H}^1(\mathcal{W}, \mathcal{O}_M)$ is of non-Hausdorff type, from which and Lemma \ref{lem:comparizon_top_nonhausdoeff_CDcorresp} the assertion follows. 
\qed

\section{Examples}

\subsection{Toroidal groups of dimension $2$}

Fix a complex number $\tau$ whose imaginary part ${\rm Im}\,\tau$ is positive. 
Denote by $C$ the elliptic curve $\mathbb{C}/\langle 1, \tau\rangle$. 
For a non-torsion element $F\in {\rm Pic}^0(C)$, 
let us consider the ruled surface $X:={\bf P}(\mathbb{I}_C\oplus F)$. 
As is mentioned in \S 1, $K_X^{-1}$ is semi-positive and $K_X^{-1}=[Y_1+Y_2]$ holds for $Y_1:={\bf P}(\mathbb{I}_C)$ and $Y_2:={\bf P}(F)$. 

Let $p, q\in [-1/2, 1/2)$ be real numbers such that the monodromy $\rho_F\colon\pi_1(C, *)\to {\rm U}(1)$ of the flat line bundle $F$ satisfies
\[
\rho_F(n+m\tau)=\exp(2\pi\sqrt{-1}(np+mq)) \quad (n, m\in\mathbb{Z}). 
\]
Then
\[
X \cong (\mathbb{C}\times \mathbb{P}^1)/\sim 
\]
holds, where the relation ``$\sim$'' is the one generated by 
\[
(z, \eta)\sim (z+\lambda, \rho_F(\lambda)\eta)\quad ((z, \eta)\in \mathbb{C}\times \mathbb{P}^1)
\]
for each $\lambda\in \langle 1, \tau\rangle$. 
Therefore, by considering the covering map 
\[
\mathbb{C}^2\ni (z, w)\mapsto \left(z,\ \exp(2\pi\sqrt{-1}\cdot w)\right)\in 
\mathbb{C}\times(\mathbb{C}\setminus\{0\}), 
\]
one has that 
\[
M \cong \mathbb{C}^2/\Lambda, \quad
\Lambda = \left\langle
\begin{pmatrix}
0 \\
1 \\
\end{pmatrix},\ 
\begin{pmatrix}
1 \\
p \\
\end{pmatrix},\ 
\begin{pmatrix}
\tau \\
q \\
\end{pmatrix}
\right\rangle
\]
holds for the complement $M:=X\setminus (Y_1\cup Y_2)$. 
Thus one has that $M$ is a toroidal group of dimension $2$. 

In this example, as is simply observed, the sufficient and necessary condition for $(i)$ in Theorem \ref{thm:main} is equivalent to the following condition: 
there exist positive numbers $\ve$ and $\delta$ such that, for any positive integer $n$, $\min\{\max\{|np-m_1|, |nq-m_2|\}\mid (m_1, m_2)\in\mathbb{Z}^2\}>\ve\cdot \delta^n$ holds. 
Thus $M$ is a toroidal theta group if this condition holds, and is a toroidal wild group if it does not hold (See \cite[2.2.2]{AK}). 
Our proof of Theorem \ref{thm:main} can be seen as an alternative proof of \cite[Theorem 4.3]{Ka} (see also \cite[Theorem 2.2.9]{AK}) for toroidal groups of dimension $2$. 

\subsection{Proof of Corollary \ref{cor:main} and the $\ddbar$-Lemma}\label{section:ex_of_blp2at9pts}

Let $X_\theta$ and $Y_\theta$ be as in Example \ref{eg:blP2at9pts_sp_by_KoTLS} for each $\theta\in \mathbb{R}$. 
As is mentioned in \S 1, it is shown that $K_{X_\theta}^{-1}$ is semi-positive for each $\theta\in \mathbb{R}$ \cite[Theorem 1.3]{KoTLS}. 
Note that there exists a configuration of nine points of $\mathbb{P}^2$ such that the anticanonical line bundle of the blow-up centered at them is not semi-positive (See \cite[Theorem 7.1, Remark 7.2]{Ko2017}). 
See also \cite{A}, \cite{U}, and \cite{B} for the study on the semi-positivity of the anticanonical line bundle of the blow-up of $\mathbb{P}^2$ at nine points. 

\begin{proof}[Proof of Corollary \ref{cor:main}]
As $X_\theta$ is a rational surface, $H^1(X_\theta, \mathcal{O}_{X_\theta})=0$ holds. 
Thus the assertion follows from Theorem \ref{thm:main} and Proposition \ref{prop:main_for_prelim_tor_case}. 
\end{proof}

\begin{remark}
From Proposition \ref{prop:hol_tub_nbhd_aru}, it follows that $Y_\theta$ admits a holomorphic tubular neighborhood in $X_\theta$ when $\theta$ is an asymptotically positive irrational number. 
Note that the existence of a holomorphic tubular neighborhood of $Y_\theta$ has been known when $\theta$ satisfies the following {\it Diophantine condition}: i.e. when there exist positive numbers $\alpha$ and $\beta$ such that, for any positive integer $n$, $\min_{m\in \mathbb{Z}}|n\theta-m|>\alpha\cdot n^{-\beta}$ \cite{A}. 
Note also that such an existence result of a holomorphic tubular neighborhood of $Y_\theta$ in this example can be applied to a gluing construction of a K3 surface, see \cite[Theorem 1.5]{KU}. 
In this context, Lemma \ref{lem:case2_kinou_lim_zero} can be regarded as a generalization of \cite[Lemma 3.1]{KU2}, which plays an essential role for constructing a projective K3 surface by the gluing method \cite{KU2}. 
\end{remark}

Let us investigate the $\ddbar$-problem on $M_\theta:=X_\theta\setminus Y_\theta$. 
In accordance with \cite{KT1} and \cite{KT2}, we say that ``{\it the $\ddbar$-lemma holds on $M_\theta$}" if the following holds: 
for any given $d$-exact $C^\infty$ $(1, 1)$-form $\vp$ on $M_\theta$, there exists a $C^\infty$ function $\Psi$ on $M_\theta$ such that $\vp=\ddbar \Psi$ holds. 
As an analogue of \cite[Example 4.4]{KT1}, we show the following: 
\begin{proposition}\label{prop:ddbar_lem_for_Bl9ptsP2}
Let $\theta$ be a real number and $(X_\theta, Y_\theta)$ be as in Example \ref{eg:blP2at9pts_sp_by_KoTLS}. Then the following holds for $M_\theta:=X_\theta\setminus Y_\theta$: \\
$(i)$ When $\theta$ is an asymptotically positive irrational number, the $\ddbar$-lemma holds on $M_\theta$. \\
$(ii)$ When $\theta$ is an asymptotically zero irrational number, the $\ddbar$-lemma does {\it not} hold on $M_\theta$. 
\end{proposition}

\begin{proof}
As the assertion $(i)$ follows from Corollary \ref{cor:main} and Lemma \ref{lem:ddbar_lem_for_H01zero} below, we will show the assertion $(ii)$ in what follows. 
Note that our argument here is based on the argument in the proof of \cite[Theorem 3.3]{KT2}. 

Assume that $\theta$ is an asymptotically zero irrational number. 
Consider the long exact sequence 
\[
H^1(M_\theta, \mathbb{Z}) \to H^1(M_\theta, \mathcal{O}_{M_\theta}) \to H^1(M_\theta, \mathcal{O}_{M_\theta}^*) \to H^2({M_\theta}, \mathbb{Z})
\]
induced from the exponential exact sequence.  
It follows from Corollary \ref{cor:main} that there exists an element $\xi\in H^1(M_\theta, \mathcal{O}_{M_\theta})$ which is not included in the image of the map $H^1(M_\theta, \mathbb{Z})\to H^1(M_\theta, \mathcal{O}_{M_\theta})$. 
Let $L$ be the image of $\xi$ by the map $H^1(M_\theta, \mathcal{O}_{M_\theta}) \to H^1(M_\theta, \mathcal{O}_{M_\theta}^*)$. 
In what follows, we regard $L$ as a holomorphic line bundle on $M_\theta$. 
Note that, by construction, $L$ is topologically trivial whereas it is not holomorphically trivial. 

Take a $C^\infty$ Hermitian metric $h$ on $L$ and denote by 
$\Theta := \sqrt{-1}\Theta_h$ the Chern curvature of it. As $L$ is topologically trivial, $[\Theta]=0\in H^2(M_\theta, \mathbb{R})$ holds. 
Therefore $\Theta$ is $d$-exact. 
Assume that there exists a function $f\colon M_\theta\to \mathbb{C}$ such that $\Theta = \sqrt{-1}\ddbar f$ holds. Then, as 
\[
\sqrt{-1}\ddbar \,\overline{f}=\sqrt{-1}\cdot  \overline{\delbar \del f}
=-\sqrt{-1} \overline{\ddbar f}=\overline{\Theta}=\Theta, 
\]
we may assume that $f$ is real-valued by replacing 
$f$ with $(f+\overline{f})/2$. 
Let $h'$ be a Hermitian metric on $L$ defined by letting $h':= h\cdot e^{-f}$. 
Then, as $\sqrt{-1}\Theta_{h'}\equiv 0$ holds by construction, it follows that $h'$ is a flat metric on $L$. 

Consider the short exact sequence 
\[
0\to {\rm U}(1) \to \mathcal{O}_{M_\theta}^*\to \mathcal{P}_{M_\theta}\to 0
\]
as in the proof of \cite[Theorem 3.3]{KT2}, where $\mathcal{P}_{M_\theta}$ denotes the sheaf of real-valued pluriharmonic functions and the map $\mathcal{O}_{M_\theta}^*\to \mathcal{P}_{M_\theta}$ is defined by 
\[
\Gamma(U, \mathcal{O}_{M_\theta}^*)\ni h\mapsto \log |h|\in 
\Gamma(U, \mathcal{P}_{M_\theta})
\]
for each open set $U\subset M_\theta$. 
Then one has the induced long exact sequence
\[
\cdots \to H^1(M_\theta, {\rm U}(1))
\to H^1(M_\theta, \mathcal{O}_{M_\theta}^*)
\to H^1(M_\theta, \mathcal{P}_{M_\theta})\to \cdots, 
\]
by which the element $L\in H^1(M_\theta, \mathcal{O}_{M_\theta}^*)$ is mapped to $0\in H^1(M_\theta, \mathcal{P}_{M_\theta})$, since $h'$ is flat. 
As $M_\theta$ is simply connected (see \cite[Chapter 3]{GS} for example), $ H^1(M_\theta, {\rm U}(1))$ is trivial. 
Therefore $L=\mathbb{I}_{M_\theta}\in H^1(M_\theta, \mathcal{O}_{M_\theta}^*)$ follows, which contradicts to the construction of $L$. Thus the $\ddbar$-lemma does not hold on $M_\theta$ in this case. 
\end{proof}

\begin{lemma}\label{lem:ddbar_lem_for_H01zero}
Let $\Omega$ be a complex manifold such that $H^1(\Omega, \mathcal{O}_\Omega)=0$. Then the $\ddbar$-lemma holds on $\Omega$. 
\end{lemma}

\begin{proof}
Take a $d$-exact $C^\infty$ $(1, 1)$-form $\vp$ on $\Omega$. 
By the $d$-exactness, there exists a $C^\infty$ $1$-form $\Phi$ on $\Omega$ such that $d\Phi=\vp$ holds. 
Let $\Phi=\Phi_{(1, 0)}+\Phi_{(0, 1)}$ be the decomposition of $\Phi$ into 
the $(1, 0)$-part $\Phi_{(1, 0)}$ and the $(0, 1)$-part $\Phi_{(0, 1)}$. 
As is simply observed, both $\overline{\Phi_{(1, 0)}}$ and $\Phi_{(0, 1)}$ are $\overline{\partial}$-closed $(0, 1)$-forms on $\Omega$. 
Therefore, by assumption, one can take $C^\infty$ function $\psi_1$ and $\psi_2$ on $\Omega$ such that $\Phi_{(1, 0)}=\overline{(\overline{\partial}\psi_1)}=\partial\overline{\psi_1}$ and $\Phi_{(0, 1)}=\overline{\partial}\psi_2$ hold. Then it holds that 
$\varphi=\partial\overline{\partial}(-\overline{\psi_1}+\psi_2)$, 
from which the assertion follows. 
\end{proof}

In the rest of this subsection, we investigate the following question, which is posed by Professor Shigeharu Takayama and concerning the Hodge decomposition theorem. 
\begin{question}\label{question_hodge_decomp_from_takayamasensei}
Let $M_\theta:=X_\theta\setminus Y_\theta$ be as in Example \ref{eg:blP2at9pts_sp_by_KoTLS}. 
When $\theta$ is an asymptotically positive irrational number, 
does the equation 
\begin{equation}\label{label:eq_hodge_eq_for_blp29pts}
{\rm dim}\,H^r(M_\theta, \mathbb{C}) 
= \sum_{p+q=r}{\rm dim}\,H^q(M_\theta, \Omega^p_{M_\theta})
\end{equation}
hold for $r=0, 1, 2$? When $\theta$ is an asymptotically zero irrational number, 
does the equation
\[
{\rm dim}\,H^r(M_\theta, \mathbb{C}) 
= \sum_{p+q=r}{\rm dim}\,H^q(M_\theta, \Omega^p_{M_\theta})_{\rm Hausdorff} 
\]
hold for $r=0, 1, 2$? 
\end{question}
Here we denote by $H^q(M_\theta, \Omega^p_{M_\theta})_{\rm Hausdorff}$ the quotient space $Z^{p, q}(M_\theta)/\overline{B^{p, q}(M_\theta)}$, which coincides with the Hausdorffization of the topological vector space $H^q(M_\theta, \Omega^p_{M_\theta})$. 
Note that the equation (\ref{label:eq_hodge_eq_for_blp29pts}) holds if the Hodge decomposition theorem holds on $M_\theta$. 
Note also that such a decomposition actually occurs on theta toroidal groups such as $M$ in the previous subsection \S 7.1 \cite[2.2.6]{AK}. 
As a partial answer to this question, we show the following: 
\begin{proposition}\label{prop:takayamasensei_question_partialsolution}
Let $\theta$ be an irrational number and $p_1, p_2, \dots$, $p_8$, $q_\theta$, and $M_\theta:=X_\theta\setminus Y_\theta$ be as in Example \ref{eg:blP2at9pts_sp_by_KoTLS}. Then the following holds: \\
$(i)$ The equation
\[
{\rm dim}\,H^r(M_\theta, \mathbb{C}) = \begin{cases}
1 & \text{if}\ r=0\\
11 & \text{if}\ r=2\\
0 & \text{otherwise}
\end{cases}
\]
holds. \\
$(ii)$ Assume that the configuration of the nine points $Z:=\{p_1, p_2, \dots, p_8, q_\theta\}$ is enough general so that $\# Z=9$, $Z$ includes four points among which no three points are collinear, and that there exists no compact analytic curve included in $M_\theta$. Assume also that 
$\theta$ is asymptotically positive. 
Then 
\[
{\rm dim}\,H^q(M_\theta, \Omega^p_{M_\theta}) = \begin{cases}
1 & \text{if}\ (p, q)=(0, 0),\ (2, 0)\\
10 & \text{if}\ (p, q)=(1, 1)\\
0 & \text{otherwise}
\end{cases}
\]
holds. 
Especially, the equation (\ref{label:eq_hodge_eq_for_blp29pts}) holds for $r=0, 1, 2$ in this case. 
\end{proposition}
%

\begin{proof}
The assertion $(i)$ directly follows from the Mayer--Vietoris sequence for the open covering $\{M_\theta, V\}$ of $X_\theta$, where $V$ is a tubular neighborhood of $Y_\theta$ in $X_\theta$. 
For $(ii)$, $H^0(M_\theta, \Omega_{M_\theta}^p)\cong \mathbb{C}$ simply follows from Lemma \ref{lem:nononconstantholfunctionexistsonV} and the triviality of the canonical bundle of $M_\theta$ for $p=0, 2$. 
The vanishing of $H^1(M_\theta, \Omega^p_{M_\theta})$ for $p=0, 2$ follows from Corollary \ref{cor:main} $(ii)$ and the triviality of the canonical bundle of $M_\theta$, and 
the vanishing of $H^2(M_\theta, \Omega^p_{M_\theta})$ for $p=0, 1, 2$ follows from the same argument as in the proof of Lemma \ref{lem:first_lem_in_prelim} $(iii)$. 
In order to investigate $H^q(M_\theta, \Omega^p_{M_\theta})$ for $(p, q)= (1, 0), (1, 1)$, 
take a holomorphic tubular neighborhood $V$ of $Y:=Y_\theta$, open covering $\{V_j\}$ of $V$, and local coordinates $(z_j, w_j)$ on each $V_j$ as in Remark \ref{rmk:niceVjs_whenCaseI_II}. 
Denote simply by $dz$ and $dw/w$ the holomorphic $1$-forms defined by $\{(V_j, dz_j)\}$ and $\{(V_j, dw_j/w_j)\}$ on $V$, respectively. 
Consider the long exact sequence 
\begin{align}\label{exactseq_MV_Mtheta}
0&\to H^0(X_\theta, \Omega_{X_\theta}^1) \to  H^0(M_\theta, \Omega_{M_\theta}^1) \oplus H^0(V, \Omega_V^1) \to H^0(V^*, \Omega_{V^*}^1) \\
&\to H^1(X_\theta, \Omega_{X_\theta}^1) \to  H^1(M_\theta, \Omega_{M_\theta}^1) \oplus H^1(V, \Omega_V^1) \to H^1(V^*, \Omega_{V^*}^1) \to 0\nonumber 
\end{align}
induced from the Mayer--Vietoris sequence for the open covering $\{M_\theta, V\}$ of $X_\theta$, where $V^*:=M_\theta\cap V (=V\setminus Y_\theta)$. 
Then, as 
\begin{equation}\label{equation:genericcase_Omega1_decomp_blp2at9pts}
\Omega_V^1=\mathbb{I}_V\cdot dz\oplus [-Y_\theta]\cdot \frac{dw}{w}
\end{equation}
holds as vector bundles, 
it follows from Lemma \ref{lem:nononconstantholfunctionexistsonV} and Lemma \ref{lem:dim_of_h01} below that there exists a exact sequence 
\begin{align*}
0 \to \mathbb{C}^9 \to H^1(M_\theta, \Omega_{M_\theta}^1) &\oplus 
\left(H^1(V, \mathcal{O}_V)\cdot dz \oplus H^1(V, I_{Y_\theta})\cdot \frac{dw}{w}\right)\\
&\to \left(H^1(V^*, \mathcal{O}_{V^*})\cdot dz \oplus H^1(V^*, \mathcal{O}_{V^*})\cdot \frac{dw}{w}\right) \to 0, 
\end{align*}
where $I_{Y_\theta}\subset \mathcal{O}_V$ is the defining ideal sheaf of $Y_\theta$. 
As it follows from 
a standard argument by using the long exact sequence induced by the short exact sequence 
$0\to I_{Y_\theta}\to \mathcal{O}_V\to \mathcal{O}_V/I_{Y_\theta}\to 0$ that 
\[
0 \to H^1(V, I_{Y_\theta}) \to H^1(V, \mathcal{O}_{V}) \to \mathbb{C} \to 0
\]
is exact, one can deduce from Proposition \ref{prop:section5_main} that 
\[
0 \to \mathbb{C}^9 \to H^1(M_\theta, \Omega_{M_\theta}^1) \to \mathbb{C} \to 0
\]
is exact, from which the assertion follows. 
\end{proof}

\begin{lemma}\label{lem:dim_of_h01}
Under the assumptions in Proposition \ref{prop:takayamasensei_question_partialsolution} $(ii)$, 
$H^0(M_\theta, \Omega^1_{M_\theta})=0$ holds. 
\end{lemma}

\begin{proof}
Here we use the notation in the proof of Proposition \ref{prop:takayamasensei_question_partialsolution}. 
Assume that there exists a non-trivial element $\xi \in H^0(M_\theta, \Omega^1_{M_\theta})$. 
By Lemma \ref{lem:nononconstantholfunctionexistsonV}, the exact sequence (\ref{exactseq_MV_Mtheta}), and the equation (\ref{equation:genericcase_Omega1_decomp_blp2at9pts}), we may assume 
\[
\xi|_{V^*} = a\cdot dz+\frac{dw}{w}
\]
holds for some $a\in \mathbb{C}$. 
Note that $\xi$ has at most isolated zeros, since there exists no compact analytic curve in $M_\theta$. 
By regarding $\xi$ as an element of $H^0(X_\theta, \Omega_{X_\theta}^1\otimes [Y_\theta])=H^0(X_\theta, \text{Hom}(T_{X_\theta}, [Y_\theta]))$ and considering its kernel, one can construct a (maybe singular) holomoprhic foliation $\mathcal{G}$ on $X_\theta$ such that the normal bundle $N_{\mathcal{G}}$ of $\mathcal{G}$ coincides with the line bundle $[Y_\theta]$ on $X_\theta$ (See \cite[Chapter 2]{Bbook} for example). 
By the equation $K_{X_\theta}^{-1}\cong T_{\mathcal{G}}\otimes N_{\mathcal{G}}$, one has that the tangent bundle $T_{\mathcal{G}}$ of $\mathcal{G}$ is a holomorphically trivial subbundle of $T_{X_\theta}$. Thus one has $H^0(X_{\theta}, T_{X_\theta})\not=0$, which contradicts to \cite[Lemma A.2]{KU} (See also the argument in the proof of \cite[Proposition 8]{B}). 
\end{proof}



\end{document}